\documentclass[11pt,a4paper]{amsart}

\usepackage[latin1]{inputenc}
\usepackage{amsmath,amsfonts,amssymb,setspace,amscd}
\usepackage{enumerate}
\usepackage[
      colorlinks=true,    
      urlcolor=blue,    
      menucolor=blue,    
      linkcolor=blue,    
      bookmarks=true,    
      bookmarksopen=true,    
      citecolor=blue,
      hyperfootnotes=false,    
      pdfpagemode=UseOutlines    
]{hyperref}
\usepackage[top=3 cm, bottom=3 cm, left=3 cm, right=3cm]{geometry}
\onehalfspace

\newtheorem{theorem}{Theorem}[section]
\newtheorem{lemma}[theorem]{Lemma}
\newtheorem{proposition}[theorem]{Proposition}

\newtheorem{corollary}[theorem]{Corollary}

\theoremstyle{definition}
\newtheorem{definition}[theorem]{Definition}
\newtheorem{remark}[theorem]{Remark}

\numberwithin{equation}{section}
\numberwithin{theorem}{section}

\newcommand{\N}{\mathbb{N}}
\newcommand{\R}{\mathbb{R}}
\newcommand{\Z}{\mathbb{Z}}
\newcommand{\Q}{\mathbb{Q}}
\newcommand{\C}{\mathbb{C}}

\newcommand{\Qbar}{\overline{\Q}}

\newcommand{\cA}{\mathcal{A}}
\newcommand{\cB}{\mathcal{B}}
\newcommand{\cC}{\mathcal{C}}
\newcommand{\End}{\text{End}}
\newcommand{\D}{\mathcal{D}}

\newcommand{\bo}[1]{\boldsymbol{#1}}
\newcommand{\mc}[1]{\mathcal{#1}}

\newcommand{\bz}{\boldsymbol{z}}
\newcommand{\per}{\boldsymbol{\omega}}
\newcommand{\bw}{\boldsymbol{w}}

\newcommand{\Kbar}{\overline{K}}

\def\lg{\left\lbrace}
\def\rg{\right\rbrace}

\author[F. Barroero]{Fabrizio Barroero}
\address{Universit\`a di Roma Tre, Dipartimento di Matematica e Fisica, Largo San Murialdo 1, 00146 Roma, Italy}
\email{fbarroero@gmail.com}

\author[L. Capuano]{Laura Capuano}
\address{Mathematical Institute,
University of Oxford,
Woodstock Road
Oxford
OX2 6GG, UK}
\email{laura.capuano1987@gmail.com}

\title[Unlikely Intersections and the polynomial Pell equation]{Unlikely intersections in families of abelian varieties and the Polynomial Pell equation}
\subjclass[2010]{11G05 - 11G50 - 11U09 - 14K05 }   
\date{\today}
\keywords{unlikely intersections, abelian varieties, Pell'e equation, o-minimality.}

\begin{document}

\begin{abstract}
Let $S$ be a smooth irreducible curve defined over a number field $k$ and consider an abelian scheme $\cA$ over $S$ and a curve $\cC$ inside $\cA$, both defined over $k$. In previous works, we proved that, when $\cA$ is a fibered product of elliptic schemes, if $\cC$ is not contained in a proper subgroup scheme of $\cA$, then it contains at most finitely many points that belong to a flat subgroup scheme of codimension at least 2. In this article, we continue our investigation and settle the crucial case of powers of simple abelian schemes of relative dimension $g\ge 2$. This, combined with the above mentioned result and work by Habegger and Pila, gives the statement for general abelian schemes which has applications in the study of solvability of almost-Pell equations in polynomials. 
\end{abstract}

\maketitle
\section{Introduction}

Let $E_\lambda$ denote the elliptic curve in the Legendre form defined by
$Y^2=X(X-1)(X-\lambda).$
In \cite{MasserZannier10} (see also \cite{MasserZannier08}), Masser and Zannier showed that there are at most finitely many complex numbers $\lambda_0 \neq 0,1$ such that the two points of coordinates $(2,\sqrt{2(2-\lambda_0)})$ and $ (3,\sqrt{6(3-\lambda_0)})$ both have finite order on the elliptic curve $E_{\lambda_0}$. Later, in \cite{MZ12} the same authors proved that one can replace $2$ and $3$ with any two distinct complex numbers ($\neq 0,1$) or even choose distinct $X$-coordinates ($\neq \lambda$) defined over an algebraic closure of $\C(\lambda)$, provided the points are not identically linearly related.
  
Results of this kind are sometimes called ``relative Manin-Mumford'' because of the clear analogy with the classical Manin-Mumford conjecture. Masser and Zannier \cite{MZpreprint} recently settled the problem for a curve in any abelian scheme, both defined over the algebraic numbers.

As a generalization of the result in \cite{MasserZannier10}, Masser (see \cite{Zannier}, p.~88) asked if there are infinitely many $\lambda_0 \in \C $ such that two independent relations with coefficients in $\Z$ between the points $(2,\sqrt{2(2-\lambda_0)})$, $(3,\sqrt{6(3-\lambda_0)})$ and $(5,\sqrt{20(5-\lambda_0)})$ hold on the same elliptic curve $E_{\lambda_0}$. 
Masser and Zannier expected a negative answer in view of very general conjectures and, indeed, this follows from the main result of \cite{BC2016}, where we proved a more general theorem about the intersection of a curve with codimension $2$ flat subgroup schemes in the $n$-fold fibered power of $E_{\lambda}$.  

The proof of this result follows the now well-established Pila-Zannier strategy, first introduced in \cite{PilaZannier} to give an alternative proof of Raynaud's Theorem \cite{Rayn} (Manin-Mumford for abelian varieties) and already used by Masser and Zannier in the works mentioned above. One of the main ingredients of their proofs is a theorem of Pila \cite{Pila04} (which is itself a generalization of \cite{BombieriPila} and was then extended in \cite{PilaWilkie} to a very general setting) to count rational points of bounded denominator on a  ``sufficiently transcendental'' real subanalytic surface.
In \cite{BC2016} instead we adapted ideas introduced in \cite{CMPZ}, where the authors counted rational points on suitable definable subsets of a Grassmannian variety to obtain an unlikely intersections result for curves in algebraic tori (see also \cite{BMZ99}). This allowed us to deal with linear relations rather than just with the torsion points.  

In this article, we push the method further and prove a statement analogous to the main theorem  of \cite{BC2016} for powers of simple abelian schemes. 
This, together with previous results \cite{HabPila14}, \cite{BC2016} and \cite{BC2017} allows us to formulate a theorem for general abelian schemes over a curve.

Fix a number field $k$ and a smooth irreducible curve $S$ defined over $k$. We consider an abelian scheme $\cA$ over $S$ of relative dimension $g\geq 2$, also defined over $k$. This means that for each $s \in S(\C)$ we have an abelian variety $\cA_s$ of dimension $g$ defined over $k(s)$. 

Let $\cC$ be an irreducible curve in $\cA$ also defined over $k$ and not contained in a proper subgroup scheme of $\cA$, even after a base extension. A component of a subgroup scheme of $\cA$ is either a component of an algebraic subgroup of a fiber or it dominates the base curve $S$. A subgroup scheme whose irreducible components are all of the latter kind is called flat.

In the previous works \cite{BC2016} and \cite{BC2017} we proved that the intersection of $\cC$ with the union of all flat subgroup schemes of $\mc{A}$ of codimension at least $2$ is finite, when $\mc{A}$ is a fibered product of elliptic schemes.
In case $\cA$ is isotrivial, the same fact was proved by Habegger and Pila \cite{HabPila14} for any abelian variety defined over the algebraic numbers.


\begin{theorem}\label{thmscheme}
Let $k$ and $S$ be as above. Let $\cA\rightarrow S$ be an abelian scheme and $\mc{C}$ an irreducible curve in $\cA$ not contained in a proper subgroup scheme of $\cA$, even after a finite base change. Suppose that $\cA$ and $\cC$ are defined over $k$. Then, the intersection of $\mc{C}$ with the union of all flat subgroup schemes of $\mc{A}$ of codimension at least $2$ is a finite set.
\end{theorem}

The above theorem follows from Conjecture 6.1 of \cite{pink}, but it does not imply the same statement for a curve in an abelian scheme. Indeed, in his conjecture, Pink considers subgroups of the fibers and these might not come from flat subgroup schemes for fibers with a larger endomorphism ring. To the authors' knowledge, Conjecture 6.1 of \cite{pink} (for the non-isotrivial case) has been settled only in the case of a curve in a fibered power of an elliptic scheme when everything is defined over $\Qbar$ (as a combination of \cite{BC2016} and \cite{BCM}), and for a curve in an abelian surface scheme, where the codimension $2$  algebraic subgroups of the fibers are torsion subgroups 
and the matter then reduces to the relative Manin-Mumford problem, settled by Masser and Zannier in a series of articles \cite{MZ12}, \cite{MZ14a}, \cite{MasserZannier15} and \cite{MZpreprint}. The first two papers deal also with the case of a complex curve in a product of complex elliptic schemes, while the same authors with Corvaja handled complex curves in complex simple abelian surface schemes in \cite{CMZ}.

If $\cA\rightarrow S$ is isotrivial, i.e., there exists a finite cover $S'\rightarrow S$ such that the generic fiber of $\cA \times_S S'\rightarrow S'$ is isomorphic to an abelian variety defined over the algebraic numbers, our Theorem \ref{thmscheme} is nothing but Theorem 1.1 of \cite{HabPila14}. Previous partial results for curves in certain abelian varieties appeared in \cite{RemVia}, \cite{Ratazzi08}, \cite{Carrizosa}, \cite{Viada2008} and \cite{galateau2010}.

Among the problems of unlikely intersections for families of abelian varieties, we shall also mention the recent result of Dill \cite{Dill}. Namely, Dill proved that, in the same setting of Theorem \ref{thmscheme}, given a fixed fiber $A_0$ of $\cA$ and a finite rank subgroup $\Gamma$ of $A_0$, there are at most finitely many points $\bo{c}\in \cC(\C)$ that lie in the image of $\Gamma$ under an isogeny, unless $\cC$ is contained in a translate of a torsion curve by a constant section of the constant part of $\cA$, generalizing an earlier result of Gao \cite{Gao17}.\\



As mentioned before, we deduce Theorem \ref{thmscheme} 
from a work of Habegger and Pila, previous works of the authors and a new theorem (Theorem \ref{maintheorem}) which is the main result of this article. In the latter we consider a simple non-isotrivial abelian scheme $\cB\rightarrow S$ of relative dimension $ g \geq2$ and an irreducible curve $\cC$ in its $n$-fold fibered power, not contained in a fixed fiber. This defines $n$ points $P_1, \dots , P_n \in \cB_\eta(k(\cC))$, where $\cB_\eta$ is the generic fiber of $\cB$. We suppose they are independent over $R$, the endomorphism ring of the generic fiber, i.e., the curve is not contained in a proper flat subgroup scheme of the ambient abelian scheme. In Theorem \ref{maintheorem} we show that there are at most finitely many points on $\cC$ such that $P_1, \dots , P_n$ become $R$-dependent after specialization. 

To prove Theorem \ref{maintheorem} we follow the Pila-Zannier strategy which has been largely employed in several works on unlikely intersections in different contexts. To deal with algebraic subgroups rather than torsion subgroups, we use ideas appeared in \cite{CMPZ} in the case of tori and in \cite{HabPila14} in the case of abelian varieties. From a general point of view, the strength  of this method lies in the fact that it allows to replace deep Dobrowolski-type or Bogomolov-type height lower bounds, needed in previous approaches to similar problems, which are known only in special cases, with weaker results that hold in more general settings.

To apply this strategy, we consider the abelian logarithms of the $P_i$ and their real coordinates in a basis for the period lattice of our abelian scheme. These give a subanalytic surface $Z$; moreover, points of $\cC$ for which the $P_i$ become dependent correspond to points on this surface lying on linear subvarieties with coefficients related to the coefficients of the relations between the $P_i$. A refinement of the Pila-Wilkie Theorem \cite{PilaWilkie}, due to Habegger and Pila \cite{HabPila14}, gives, for arbitrary $\epsilon>0$, an upper bound of order $T^{\epsilon}$ for the number of points of $Z$ lying on subspaces of the special form mentioned above and rational coefficients of height at most $T$, provided our abelian logarithms are algebraically independent over the field generated by the period basis. This is ensured by a result of Bertrand \cite{Bert11}. 

Now, to conclude the proof, we combine results of David \cite{David93}, Masser \cite{Masser88}, Masser-W\"ustholz \cite{MassWu}, Masser-Zannier \cite{MZpreprint}, Pazuki \cite{Pazu} and Silverman \cite{Sil83} (who gives a bound on the height of the points on $\cC$ we are considering) to show that the number of points on $Z$ considered above is of order at least $T^\delta$ for some $\delta>0$. Comparing the two estimates leads to an upper bound for $T$ and thus for the coefficients of the relation between the $P_i$, concluding the proof. \\


The second part of the paper is devoted to applications of Theorem \ref{thmscheme}. In Section \ref{aux} we deduce from Theorem \ref{thmscheme} an apparently stronger result (Theorem \ref{thmm}) which we state here.
As before we consider an abelian scheme $\cA$ over a smooth irreducible curve $S$, of relative dimension $g$ and an irreducible curve $\cC\subseteq \cA$, not contained in a fixed fiber. Everything is defined over a fixed number field $k$.

\begin{theorem} \label{thmaux}
	Let $m$ be an integer, $1\leq m \leq g$. If the intersection of $\cC$ with the union of all flat subgroup schemes of $\cA$ of codimension at least $m$ is infinite, then there exists a finite cover $S'\rightarrow S$ such that $\cC\times_S S'$ is contained in a flat subgroup scheme of $\cA\times_S S'$ of codimension at least $m-1$.
\end{theorem}

Note that, for $m=1$ the statement is trivial, for $m=2$ this is just a reformulation of Theorem \ref{thmscheme} and if $m=g$ this is Theorem 1.7 of \cite{MZpreprint}.

We will give two applications of the above statement. The first is a Mordell-Lang-type statement (Theorem \ref{thmSG}) which somehow resembles a recent result of Ghioca, Hsia and Tucker \cite{Ghioca17}. In the Appendix we show how their Theorem 1-1 can be deduced from the main results of \cite{BC2016} and \cite{BC2017}.

As a second application of Theorem \ref{thmaux} suggested by Zannier, we consider a function field variant of the classical Pell equation. The application of unlikely intersections results to the study of solvability of polynomial Pell equations was introduced for the first time by Masser and Zannier in \cite{MasserZannier15}. 

Let $D$ be a positive integer. As it is commonly known, the classical Pell equation is an equation of the form $A^2-DB^2=1$, to be solved in integers $A,B\neq 0$. It is a classical theorem of Lagrange that such an equation is non-trivially solvable if and only if $D$ is not a perfect square. 

To obtain a polynomial analogue, we replace $\Z$ with $K[X]$, for $K$ a field to be specified later. For a non-constant $D=D(X)\in K[X]$, one looks for solutions of
$$
A^2-DB^{2}=1,
$$
where $A(X),B(X)\in \Kbar[X]$ and $B\neq 0$.  We call a possible solution with $B=0$ trivial. 

The matter in the polynomial case is more complicated, and depends heavily on the choice of the field $K$. 
In this paper, we consider fields of characteristic $0$ and we call \textit{Pellian} the polynomials $D(X)$ such that the associated Pell equation has a non-trivial solution in $\overline{K}[X]$.  Moreover, we will always assume that $D$ is squarefree. 

A necessary condition for $D(X)$ to be Pellian is that $D(X)$ is not a square and has even degree $2d$. However, unlike the classical case, these conditions are also sufficient only if $D$ has degree $2$. For higher degrees there are examples of polynomials that satisfy these necessary conditions but are not Pellian (see \cite{Zannier_indam}). 

The problem of finding a solution of the Pell equation $A^2-DB^2=1$ translates to establishing whether a certain point has finite or infinite order in the Jacobian of the hyperelliptic curve $Y^2=D(X)$. Hence, applying results of unlikely intersection type, one can deduce results about the Pellianity of polynomials varying in a one-parameter family. Indeed, in \cite{MasserZannier15} Masser and Zannier investigated the problem for the one-dimensional family $D_t(X)=X^6+X+t$. Clearly, if the family were identically Pellian, then $D_{t_0}$ would be Pellian for every specialization $t_0 \in \C$.
However, it can be proved that this is not the case, and actually there are only finitely many complex $t_0$ for which the polynomial $X^6+X+t_0$ is Pellian (for example $t_0=0$). 



Of course there is nothing special about the family $X^6+X+t$, and in fact the same finiteness result is true for any non-identically Pellian squarefree $D \in \overline{\Q(t)}[X]$, of even degree at least $6$ and such that the Jacobian $J_D$ of the hyperelliptic curve $Y^2=D(X)$ does not contain a one-dimensional abelian subvariety, as proved by the same authors in \cite{MZpreprint}. In the same paper, the authors study also the case when $J_D$ contains an elliptic curve. Related results for non-squarefree $D$ appear in \cite{BMPZ} and \cite{Schm}. In this case, one must consider the generalized Jacobian of the singular curve $Y^2=D(X)$. \\

In \cite{MasserZannier15}, Masser and Zannier also studied equations of the form
$ A^2-(X^6+X+t)B^2=c'X+c$. More specifically, they proved that there are infinitely many complex $t_0$ for which there exist $A$ and non-constant $B$ in $\C[X]$ and $c'\neq 0$, $c$ in $\C$ satisfying the above equation.

If one instead fixes $c'$ and $c$, then finiteness is expected to hold. We are going to prove that it follows from our main result.

More generally, we consider a squarefree polynomial $D(X)\in K[X]$ of degree $2d>2$ and a non-zero polynomial $F(X)\in K[X]$. We are interested in the non-trivial solutions of the \textit{``almost-Pell equation''}, i.e.,
\begin{equation} \label{almostPell}
A^2-DB^2=F,
\end{equation}
where $A,B\in K[X]$ and $B\neq 0$. Note that this equation can have a trivial solution only if $F$ is a square.

Let $S$ be a smooth irreducible curve defined over a number field $k$ and $K$ be its function field $k(S)$. As before, if the equation 
\eqref{almostPell} is identically solvable (over $\overline{K}$), then it will remain solvable after specialization at every point $s_0 \in S(\C)$, except possibly for a Zariski-closed proper subset, and the solutions will be nothing but the specializations of a general solution. On the other hand, if it is not identically solvable, then we can still have points $s_0\in S(\C)$ such that the specialized equation
$A^2-D_{s_0}B^2=F_{s_0}$ has a solution $A,B$ in $\C[X]$ with $B\neq 0$, where we denote by $D_{s_0}$ and $F_{s_0}$ the polynomials in $k(s_0)[X]$ obtained specializing the coefficients of $D$ and $F$ in $s_0$. 

Again, the existence of a non-trivial solution translates to the existence of certain relations between particular points on the Jacobian $J_D$ of the hyperelliptic curve defined by $Y^2=D(X)$. Our Theorem \ref{thmscheme} allows us to deduce the following result.

\begin{theorem}\label{Pellthm}
Let $S$ be as above, $K=k(S)$ and let $D,F \in K[X]$ be non-zero polynomials. Assume that $D$ is squarefree and has even degree at least 6 and that $J_D$ contains no one-dimensional abelian subvariety over $\overline{K}$. Then, if the equation $A^2-DB^2=F$ does not have a non-trivial solution in $\Kbar[X]$, there are at most finitely many $s_0\in S(\C)$ such that the specialized equation $A^2-D_{s_0}B^2=F_{s_0}$ has a solution $A, B \in\C[X]$, $B\neq 0$.
\end{theorem} 

If $F$ is a non-zero constant in $K$, then we fall in the case of the Pell equation and the above theorem, as already mentioned above, is a consequence of Theorem 1.3 of \cite{MZpreprint}.

Let us finally see an example. For details and more examples we refer to Section \ref{Pell_example}. Let us consider the family defined by $D_t(X)=(X-t)(X^7-X^3-1)$ and $F(X)=4X+1$. The Jacobian of the hyperelliptic curve $Y^2=D_t(X)$ is identically simple, and it can be shown that the equation $A^2-D_tB^2=F$ has no solution in $\overline{\Q(t)}[X]$. Then, by Theorem \ref{Pellthm}, there are at most finitely many $t_0\in \C$ such that the specialized equation $A^2-D_{t_0}B^2=F$ has a solution $A, B \in\C[X]$, $B\neq 0$. For instance, for $t_0=0$, one has
$$ (2X^4+1)^2-X(X^7-X^3-1)2^2=4X+1.  $$ 

\section{Reduction to powers of simple abelian schemes} \label{proofschemes}

In this section, we use Poincar\'e Reducibility Theorem to reduce Theorem \ref{thmscheme} to four cases. Three of them have been dealt with in earlier works of the authors \cite{BC2016}, \cite{BC2017} and in the work of Habegger and Pila \cite{HabPila14}, while the fourth is considered in Theorem \ref{maintheorem} below. Part of this section is inspired by \cite{Hab}.

We recall our setting. We consider an abelian scheme $\mc{A}$ over a smooth irreducible curve $S$ and everything is defined over $\Qbar$. We call $\pi:\mc{A}\rightarrow S$ the structural morphism.

Note that, since we are proving a finiteness result, we are always allowed to replace $S$ by a non-empty open subset and we tacitly do so. This allows us to pass from an abelian variety defined over a function field of a curve to the corresponding abelian scheme over (a non-empty open subset of) the curve.

A subgroup scheme $G$ of $\mc{A}$ is a closed subvariety, possibly reducible, which contains the image of the zero section $S\rightarrow \mc{A}$, is mapped to itself by the inversion morphism and such that the image of $G\times_S G$ under the addition morphism is in $G$.
A subgroup scheme $G$ is called \emph{flat} if $\pi_{|_G} : G \rightarrow S$ is flat. By  \cite{Hart}, Proposition III.9.7, as $S$ has dimension $1$, this is equivalent to ask that all irreducible components of $G$ dominate the base curve $S$. 

Now we want to show that, in order to prove Theorem \ref{thmscheme}, we can perform finite base changes and isogenies.

\begin{lemma}\label{lemmabasechange}
Let $\cC$ be as in the hypotheses of Theorem \ref{thmscheme}. Let $\cA'=\cA\times_S S'$ for some finite cover $S'\rightarrow S$ and $f$ be the projection $\cA'\rightarrow\cA$. Then, if the claim of Theorem \ref{thmscheme} holds for all irreducible components of $f^{-1}(\cC)$ it holds for $\cC$.
\end{lemma}

\begin{proof}
First, we see that $f$ is flat because it is the product of two flat morphisms $\cA\rightarrow \cA$ and $S' \rightarrow S$. Moreover, since $\cA$ and $\cA'$ have the same dimension, by \cite{Hart} Corollary III.9.6, it follows that $f$ is quasi-finite, and finite, since it is also proper. By Corollary III.9.6 of \cite{Hart}, we have that if $X \subseteq \cA$ is an irreducible variety dominating $S$, as $f$ is finite and flat, each component of $f^{-1}(X)$ is a variety of the same dimension dominating $S'$.
It is clear now that if the hypotheses of Theorem \ref{thmscheme} hold for $\cC$ then they must hold for all components of $f^{-1}(\cC)$. Finally, the preimage of any point of $\cC$ lying in a flat subgroup scheme of $\cA$ of codimension at least 2 must lie in a flat subgroup scheme of the same codimension.
\end{proof}

\begin{lemma} \label{lemisog}
Let $\cA$ and $\cA'$ be abelian schemes over the same curve $S$ and let $f_\eta:\cA'_\eta \rightarrow \cA_\eta$ be an isogeny between the generic fibers defined over $k(S)$. Moreover, let $\cC\subseteq \cA$ be a curve satisfying the hypotheses of Theorem \ref{thmscheme}. Then, if the claim of Theorem \ref{thmscheme} holds for all irreducible components of $f^{-1}(\cC)$ it holds for $\cC$.
\end{lemma}

\begin{proof}
Since every abelian scheme is a N\'eron model of its generic fiber (see \cite{Neron} Proposition 8, p.~15), there exists a map $f:\cA'\rightarrow \cA$ extending the isogeny $\cA_\eta'\rightarrow \cA_\eta$. We first show that $f$ is finite and flat. Since $f_\eta : \cA_\eta'\rightarrow \cA_\eta$ is an isogeny, we know there exists an isogeny $h_\eta : \cA_\eta\rightarrow \cA_\eta'$ such that, for some positive integer $d$, the compositions $f_\eta \circ h_\eta$ and $h_\eta \circ f_\eta$  are the multiplication-by-$d$ endomorphisms on $\cA_\eta$ and $\cA_\eta'$, respectively. Now, such morphisms uniquely extend to the multiplication-by-$d$ maps (which we just indicate by $d$) on the whole schemes. Therefore, for all $s \in S$, we have that $f_s \circ h_s$ and $h_s \circ f_s$ are finite surjective morphisms and so $f_s$ is a finite and surjective morphism between non-singular varieties. Moreover, $f_s$ is flat and by \cite{Neron}, Proposition 2, p.~52, also $f$ must be flat. Now, $f$ must be quasi-finite and also proper because $\cA' \rightarrow \cA \rightarrow S$ is proper and therefore $f$ is finite. Finally, since all $f_s$ are isogenies, the map $f$ must respect the group law of $\cA'$ and $\cA$. 
Now the claim follows as in the previous lemma since after a base change we have that images and preimages via $f$ of flat subgroup schemes are contained in flat subgroup schemes of the same dimension.
\end{proof}

Now we want to show that it is enough to prove Theorem \ref{thmscheme} for products of simple abelian schemes. Consider the generic fiber $\cA_\eta$ of $\cA$ as an abelian variety defined over $k(S)$. It is well known that every abelian variety is isogenous to a product of simple abelian varieties, see for instance \cite{Hindry2000}, Corollary A.5.1.8. Therefore there are geometrically simple pairwise non-isogenous abelian varieties $B_1, \dots , B_m$, such that $\cA_\eta$ is isogenous to $A':=\prod_{i} B_i^{n_i}$. Note that $A'$ and the isogeny might not be defined over $k(S)$ but over a finite extension which has the shape $k(S')$ for some irreducible, non-singular curve $S'$ covering $S$. By Lemma \ref{lemmabasechange} we can suppose $S'=S$. For the same reason we can suppose that the endomorphism ring of $A'$ is defined over $k(S)$. The abelian varieties $B_i$ will extend to abelian schemes $\cB_i \rightarrow S$ and we define $\cA'$ to be the fibered product over $S$ of the $n_i$-th fibered powers of the $\cB_i$. We can now apply Lemma \ref{lemisog} and we are reduced to proving Theorem \ref{thmscheme} for products of simple abelian schemes. \\

We are now going to describe flat subgroup schemes of $\cA$, which is a fibered product $\mc{A}_1 \times_S \dots \times_S \mc{A}_m$ where $\mc{A}_i$ is the $n_i$-th fibered power of $\mc{B}_i$ and $\cB_1, \dots , \cB_m$ are abelian schemes whose generic fibers are pairwise non-isogenous geometrically simple abelian varieties. Moreover, we let $R_i$ be the endomorphism ring of $B_i$ which we can suppose to be defined over $k(S)$.

Fix $i_0$, with $1\leq i_0 \leq m$. For every $\bo{a}=(a_1, \dots , a_{n_{i_0}})\in R_{i_0}^{n_{i_0}} $ we have a morphism $\bo{a}:\mc{A}_{i_0} \rightarrow \mc{B}_{i_0}$ defined by
$$
\bo{a}(P_1,\dots, P_{n_{i_0}})=a_1 P_1 +\dots +a_{n_{i_0}} P_{n_{i_0}}.
$$
We identify the elements of $R_{i_0}^{n_{i_0}}$ with the morphisms they define. The fibered product $\bo{a}_1 \times_S \dots \times_S \bo{a}_r$, for $\bo{a}_1 , \dots , \bo{a}_r \in R_{i_0}^{n_{i_0}}$ defines a morphism $\mc{A}_{{i_0}} \rightarrow \mathcal{A}''$ over $S$ where $\mathcal{A}''$ is the $r$-fold fibered power of $\cB_{{i_0}}$. Therefore, square matrices with entries in $R_{i_0}$ and appropriate size will define endomorphisms of $\mc{A}_{i_0}$.
Finally, every endomorphism of $\cA$ is given by fibered products of such endomorphisms, and is represented by an $m$-tuple in $\prod \text{Mat}_{n_i}(R_i)$ which forms a ring we call $R$.

If $\alpha\in R$, the kernel of $\alpha$ is the fibered product of $\alpha: \cA \rightarrow \cA$ with the zero section $S\rightarrow \cA$. We will denote it by $\ker \alpha$ and we will consider it as a closed subscheme of $\cA$. Let $g_i$ be the relative dimension of $\cB_i$ over $S$. If $\alpha=(\alpha_1, \dots , \alpha_m)\in R$, we define the rank $r(\alpha)$ of $\alpha$ to be $\sum \text{rank}(\alpha_i)g_i$.
In practice, $\ker \alpha$ is a closed subscheme of $\cA$ obtained by imposing $\text{rank}(\alpha_i)$ independent equations on each factor $\cA_i$.

\begin{lemma}\label{lemmat}
	Let $G$ be a flat subgroup scheme of $\mc{A}$ of codimension $d$. Then, there exists an $\alpha \in R$ of rank $d$ such that $G\subseteq \ker \alpha$ and, for any $\alpha$ of rank $d$, $\ker \alpha$ is a flat subgroup scheme of codimension $d$.
\end{lemma}

\begin{proof}
	The lemma can be proved following the line of the proof of Lemma 2.5 of \cite{Hab}. The fact that there is an $s\in S(\C)$ such that the endomorphism ring of $\cA_s$ is exactly $R$ follows from Corollary 1.5 of \cite{Noot}.
\end{proof}

From this lemma we can deduce that each flat subgroup scheme of $\cA$ is contained in a flat subgroup scheme of the same dimension and of the form
$$
G=G_1\times_S \dots \times_S G_m, 
$$
where, for every $i=1, \ldots, m$, $G_i$ is a flat subgroup scheme of $\mc{A}_i$. Now, we are interested in flat subgroup schemes of codimension at least 2. If $g_{i_0}\geq 2 $, then any proper flat subgroup scheme of $\mc{A}_{i_0}$ has codimension at least 2. This implies that any flat subgroup scheme of $\cA$ of codimension at least 2 is contained in a $G=G_1\times_S \dots \times_S G_m$ of codimension at least 2 where all $G_i = \mc{A}_i$ except for one index $i_0$ or two indexes $i_1, i_2$ with $g_{i_1}=g_{i_2}=1$. 
It is then clear that, by projecting on the factors, we only need to prove our Theorem \ref{thmscheme} in the following cases:
\begin{enumerate}
	\item $\mc{A}$ is isotrivial, i.e., it is isomorphic to a constant abelian variety after a finite base change;
	\item $\mc{A}$ is not isotrivial, $m=1$ and $g_1=1$;
	\item $\mc{A}$ is not isotrivial, $m=2$ and $g_1=g_2=1$;
	\item $\mc{A}$ is not isotrivial, $m=1$ and $g_1\geq 2$.
\end{enumerate}

In the first three cases the statement of Theorem \ref{thmscheme} follows respectively from  
\begin{enumerate}
	\item the work of Habegger and Pila \cite{HabPila14};
	\item Theorem 2.1 of \cite{BC2016};
	\item Theorem 1.3 of \cite{BC2017}.
\end{enumerate}

In case (4), we have a non-isotrivial abelian scheme $\mc{A} \rightarrow S$ which is the $n$-th fold fibered power of a simple abelian scheme $\cB$ of relative dimension at least 2.
As above, $\pi $ indicates the structural morphism $\mc{A} \rightarrow S$.

Now, since our irreducible curve $\cC$ in $\cA$, also defined over $k$, is not contained in a fixed fiber, it defines a point in $\cA_\eta(k(\cC))$ or, equivalently, $n$ points $P_1, \dots , P_n \in \cB_\eta(k(\cC))$, while for any $\bo{c} \in \cC(\C)$ we have a specialized point of $\cA_{\pi(\bo{c})}(k(\bo{c}))$ or $n$ points $P_1(\bo{c}), \dots , P_n(\bo{c}) \in \cB_{\pi(\bo{c})}(k(\bo{c}))$.

Let $R$ be the endomorphism ring of $\cB_\eta$, which we supposed to be defined over $k(S)$. Every element of $R$ specializes to an element of $\End(\cB_s)$ and this specialization map is injective. By abuse of notation we indicate by $R$ the specializations of $\End(\cB_\eta)$. Note that for some $s$ one may have $R\subsetneq \End(\cB_s)$.

The points $P_1, \dots , P_n$ defined by $\cC$ might or might not satisfy one or more relations of the form
$$
\rho_1 P_1 + \dots + \rho_n P_n =O,
$$
for some $\rho_1, \dots , \rho_n \in R$, not all zero, where $O$ is the origin of $\cB_\eta$. If they do, then clearly the same relations hold for all specializations $P_1(\bo{c}), \dots ,P_n(\bo{c})$. On the other hand, for a specific $\bo{c}$, some new relations might arise, with coefficients in $R$ or in the possibly larger $\End\left(\cB_{\pi(\bo{c})}\right)$.

As we have seen above, flat subgroup schemes correspond to linear relations over $R$ so we consider the case in which no generic relation holds and prove that there are at most finitely many specializations such that the points satisfy a relation with coefficients in $R$.

The following theorem deals with case (4) and completes the proof of Theorem \ref{thmscheme}.

\begin{theorem}\label{maintheorem}
Let $\cA\rightarrow S$ and $\cC$ be as above. Suppose that the points $P_1, \dots , P_n$ defined by $\cC$ are $R$-independent and that $\cC$ is not contained in a fixed fiber. Then, there are at most finitely many $\bo{c} \in \cC(\C)$ such that there exist $\rho_1, \dots , \rho_n \in R$, not all zero, with
$$
\rho_1 P_1(\bo{c}) + \dots + \rho_n P_n(\bo{c}) =O,
$$
on $\cB_{\pi(\bo{c})}$.
\end{theorem}

In case $n=1$ then one has a single point which is not generically torsion. There are at most finitely many specializations such that the point is torsion. As already mentioned in the Introduction, this was proved for $g=2$ by Masser and Zannier in \cite{MasserZannier15} and jointly with Corvaja in \cite{CMZ} when everything is defined over $\C$, while the case of arbitrary $g$ has been considered in \cite{MZpreprint}, again with the abelian scheme and the curve defined over the algebraic numbers.

It is easy to see that Theorem \ref{maintheorem} allows us to deduce Theorem \ref{thmscheme} in case (4).
Indeed, by Lemma \ref{lemmat}, a point $\bo{c}$ is contained in a flat subgroup scheme of codimension $\geq 2$ if and only if there is a non-trivial $R$-relation between $P_1 (\bo{c}), \dots ,P_n (\bo{c}) $ and $\cC$ is not contained in a proper subgroup scheme if and only if it is not contained in a fixed fiber and $P_1, \dots , P_n$ are generically $R$-independent. This concludes the proof of Theorem \ref{thmscheme}.

\section{O-minimality and point counting}

For the basic properties of o-minimal structures we refer to \cite{vandenDries1998} and \cite{DriesMiller}.

\begin{definition}
A \textit{structure} is a sequence $\mathcal{S}=\left( \mathcal{S}_N\right)$, $N\geq 1$, where each $\mathcal{S}_N$ is a collection of subsets of $\R^N$ such that, for each $N,M \geq 1$:
\begin{enumerate}
\item $\mathcal{S}_N$ is a boolean algebra (under the usual set-theoretic operations);
\item $\mathcal{S}_N$ contains every semialgebraic subset of $\R^N$;
\item if $A\in \mathcal{S}_N$ and $B\in \mathcal{S}_M$ then $A\times B \in \mathcal{S}_{N+M}$;
\item if $A \in \mathcal{S}_{N+M}$ then $\pi (A) \in \mathcal{S}_N$, where $\pi :\R^{N+M}\rightarrow \R^N$ is the projection onto the first $N$ coordinates.
\end{enumerate}
If $\mathcal{S}$ is a structure and, in addition,
\begin{enumerate}
\item[(5)] $\mathcal{S}_1$ consists only of all finite unions of open intervals and points,
\end{enumerate}
then $\mathcal{S}$ is called an \textit{o-minimal structure}.
\end{definition}
Given a structure $\mathcal{S}$, we say that $S \subseteq \R^N$ is a \textit{definable} set if $S\in \mathcal{S}_N$. 

Let $U\subseteq \R^{M+N}$ and let $\pi_1$ and $\pi_2$ be the projection maps on the first $M$ and on the last $N$ coordinates, respectively. Now, for $t_0\in \pi_1(U)$, we set $U_{t_0}=\{ x\in \R^N: (t_0,x) \in U \}$ and call $U$ a \emph{family} of subsets of $\R^N$, while $U_{t_0}$ is called the \emph{fiber} of $U$ above $t_0$. If $U$ is a definable set, then we call it a \emph{definable family} and one can see that the fibers $U_{t_0}$ are definable sets too. Let $S\subseteq \R^N$ and $f:S\rightarrow \R^M$ be a function. We call $f$ a \emph{definable function} if its graph $\lg (x,y) \in S\times \R^{M}:y=f(x) \rg$ is a definable set. It is not hard to see that images and preimages of definable sets via definable functions are still definable.

There are many examples of o-minimal structures, see \cite{DriesMiller}. In this article we are interested in the structure of \textit{globally subanalytic sets}, usually denoted by $\R_{\text{an}}$. We are not going to pause on details about this structure because it is enough for us to know that if $D\subseteq \R^N$ is a compact definable set, $I$ is an open neighbourhood of $D$ and $f:I\rightarrow \R^M$ is an analytic function then $f(D)$ is definable in $\R_{\text{an}}$. The fact that $\R_{\text{an}}$ is o-minimal follows from the work of Gabrielov \cite{Gabrielov}.

\begin{proposition}[\cite{DriesMiller}, 4.4] \label{unifbound}
Let $U$ be a definable family. There exists a positive integer $\gamma$ such that each fiber of $U$ has at most $\gamma$ connected components.
\end{proposition}

We now need to define the height of a rational point. The height used in \cite{HabPila14} is not the usual projective Weil height, but a coordinatewise affine height. If $a/b$ is a rational number written in lowest terms, then $H(a/b)= \max \{|a|,|b|\}$ and, for an $N$-tuple $(\alpha_1, \dots , \alpha_N) \in \Q^N$, we set $H(\alpha_1, \dots , \alpha_N)= \max H(\alpha_i)$.
For a family $Z$ of $\R^{M+N}$ and a positive real number $T$ we define
\begin{equation}\label{def}
Z^{\sim}(\Q,T)=\lg (y,z)  \in  Z: y\in \Q^M, H(y) \leq T \rg.
\end{equation}
We let $\pi_1$ and $\pi_2$ be the projection maps from $Z$ to the first $M$ and last $N$ coordinates respectively.

The following is a consequence of Corollary 7.2 of \cite{HabPila14}. 

\begin{proposition}[\cite{HabPila14}, Corollary 7.2]\label{HabPila}
For every $\epsilon>0$, then there exists a constant $c=c(Z,\epsilon)$ with the following property. If $T\geq 1 $ and $|\pi_2(Z^{\sim}(\Q,T))|> c T^\epsilon$, then there exists a continuous definable function $\delta:[0,1] \rightarrow Z$ such that
\begin{enumerate}
\item the composition $\pi_1 \circ \delta : [0,1] \rightarrow \R^M$ is semi-algebraic and its restriction to $(0,1)$ is real analytic;
\item the composition $\pi_2 \circ \delta : [0,1] \rightarrow \R^N$ is non-constant.
\end{enumerate}
\end{proposition}

\section{Abelian integrals and periods} \label{abintper}

In this section we give some definitions and facts about abelian integrals and periods. These will be used to define the set we will apply Proposition \ref{HabPila} to. For more details we refer to the Appendix of \cite{BerPil} and Section 2 of \cite{CMZ}.

We remove from $\cC$ the singular points and the ramified points of $\pi_{|_\cC}$. We call $\widehat{\cC}$ what is left. Moreover, we set $K=\C(\cC)$.

Now, let $\bo{c}^* \in \widehat{\cC}$ and consider a small neighbourhood $ N_{\bo{c}^*}$ of $\bo{c}^*$ in $\widehat{\cC}$, mapping injectively to $S$ via $\pi$. Let $D_{\bo{c}^*}$ be a subset of $\pi( N_{\bo{c}^*})$ containing $\pi(\bo{c}^*)$ and analytically isomorphic to a closed disc.

Our simple abelian scheme $\cB\rightarrow S$ defines an analytic family $\cB^{an}$ of Lie groups over the Riemann surface $S^{an}$ and its relative Lie algebra $(Lie ~\cB)/S$ defines an analytic vector bundle $Lie~\cB^{an}$ over $S^{an}$, of rank $g$. Over $D_{\bo{c}^*}$ we have a local system of periods $\Pi_\cB$ of $\cB^{an}/D_{\bo{c}^*}$ given by the kernel of the exponential exact sequence
$$
0 \longrightarrow \Pi_\cB \longrightarrow Lie~\cB^{an} \xrightarrow{\exp_\cB}  \cB^{an} \longrightarrow 0
$$
over $S^{an}$.

Possibly restricting to a non-empty open subset of $S$, we fix a basis defined over $K$ of the $K$-vector space $Lie ~\cB$ and this gives us a trivialization $Lie~\cB^{an}/D_{\bo{c}^*}\simeq D_{\bo{c}^*}\times \C^g$.

Since $D_{\bo{c}^*}$ is simply connected, we can choose $2g$ holomorphic functions $\bo{\omega}_1, \dots, \bo{\omega}_{2g} :D_{\bo{c}^*} \rightarrow \C^{g} $ such that, for every $s \in D_{\bo{c}^*}$, we have that $\bo{\omega}_1 (s), \dots, \bo{\omega}_{2g}(s)$ is a basis for the period lattice $\Pi_{\cB_s}$. Moreover, our points $P_1, \dots , P_n$ correspond to regular sections $\tilde P_i: D_{\bo{c}^*} \rightarrow \cB^{an}$ for all $i=1, \ldots, n$ and we can define holomorphic functions
$\bo{z}_1, \dots, \bo{z}_n :D_{\bo{c}^*} \rightarrow \C^{g} $ such that $\exp_{\cB_{s}}(\bo{z}_i(s))=\tilde P_i(s)$ for all $s \in D_{\bo{c}^*}$ and $i=1, \dots ,n$.

The following Lemma is a consequence of work of Bertrand \cite{Bert11}.

\begin{lemma}\label{lemtransc}
Let $F=K(\bo{\omega}_1, \dots ,\bo{\omega}_{2g})$. Under the hypotheses of Theorem \ref{maintheorem} we have
$$
\textsl{tr.deg.}_F F (\bo{z}_1, \dots ,\bo{z}_n )=ng,
$$
on $D_{\bo{c}^*}$.
\end{lemma}

\begin{proof}
This is a special case of Theorem 4.1 of \cite{Bert11} (see also \cite{Bertr09}) with $\overline{x}=(\bo{z}_1, \dots ,\bo{z}_n )$ and $\overline{y}=(\tilde P_1,\dots , \tilde P_n)$. Indeed, we have no constant part and $\overline{y}$ is non-degenerate since there is no relation among the $P_i$. Finally, we can choose $F$ as base field because of Theorem 4.3.
\end{proof}


\section{Points lying on rational linear varieties}

As before, we denote by $\widehat{\cC}$ the points of $\cC$ which are not singular nor ramified points of $\pi_{|_\cC}$.
Fix a $\bo{c}^* \in \widehat{\cC}$, and, as in the previous section, consider $ N_{\bo{c}^*}$, a small neighbourhood of $\bo{c}^*$, mapping injectively to $S$ via $\pi$. Moreover, let $D_{\bo{c}^*}$ be a subset of $\pi( N_{\bo{c}^*})$ containing $\pi(\bo{c}^*)$ and analytically isomorphic to a closed disc. For $i=1, \ldots, n$ let $\tilde P_i: D_{\bo{c}^*} \rightarrow \cB^{an}$ denote the regular sections corresponding to the points $P_i$. For the rest of this section we suppress the dependence on these data in the notation. Every constant will anyway depend on the choice of them. 

Let $\varphi_1, \dots ,\varphi_r$ be generators of $R$ as a $\Z$-module. Suppose that, for an $s_0 \in D$, there are $\rho_1, \dots, \rho_n$, not all zero, with
$$
\rho_1 \tilde P_1(s_0) +\dots + \rho_n \tilde P_n(s_0)=O
$$
on $\cB_{s_0}$, and 
$$
\rho_i = \sum_{j=1}^r a_{i,j} \varphi_j,
$$
for some integers $a_{i,j}$. Then we have
\begin{equation}\label{relD}
\sum_{i,j}a_{i,j}\varphi_j \tilde P_i(s_0) =O.
\end{equation}

For $T\geq 1$ we define
\begin{equation}\label{DefD(T)}
D(T)=\{ s_0 \in D: \eqref{relD} \text{ holds for some $a_{i,j}\in \Z^{rn}\setminus \{0\}$ and } |a_{i,j}|\leq T   \}.
\end{equation}

In this section and in the following ones we use Vinogradov's $\ll$ notation: for two real valued functions $f$ and $g$ we write $f\ll g$ if there exists a positive constant $c$ so that $f\leq c g$. At the beginning of each section we specify what these implied constants depend on. Any further dependence is denoted by an index.
Here, they depend on $\cC$, $D$, and the choice of generators of $R$.

\begin{proposition}\label{propest}
	Under the hypotheses of Theorem \ref{maintheorem}, for every $\epsilon >0$, we have $|D(T)|\ll_\epsilon T^\epsilon$, for every $T\geq 1$.
\end{proposition}

To prove this Proposition we need some preliminary lemmas. First of all, note that each endomorphism $\varphi_j$, for $j=1, \ldots, r$, is represented by a square matrix $A_j $ of dimension $g$, i.e., if for every $i=1, \ldots, n$ we set $\bw_{i,j}=A_j \bz_i$, then $\exp_{\cB_{s_0}}(\bw_{i,j}(s_0) )=\varphi_j \tilde P_i(s_0)$, for all $s_0 \in D$. We need to know more about the entries of these matrices.

\begin{lemma}\label{lementrM}
Any matrix $M$ associated to an endomorphism of $\cB_{\eta}$ has entries in $K$.
\end{lemma}

\begin{proof}
Recall that we supposed that the endomorphisms of $\cB_{\eta}$ are defined over $K$. The matrix $M$ associated to an endomorphism is the matrix representation of its differential at $O$. Then, as we have chosen a $K$-basis of $Lie~\mc{B}$ the entries of this matrix must lie in $K$.
\end{proof}

Note that \eqref{relD} implies that
$$
\sum_{i,j}a_{i,j} \bw_{i,j}(s_0) \in \Z \per_1(s_0) + \dots + \Z \per_{2g}(s_0).
$$

Recall that $D$ is a subset of $S(\C)$ analytically isomorphic to a closed disc. We now identify it with a closed disc in $\R^2$. On a small neighborhood $I$ of $D$ we can define $2g\cdot nr$ real analytic functions $u_{i,j}^{h}$ by the equations
$$
\bw_{i,j}=\sum_{h=1}^{2g} u_{i,j}^{h} \per_h,
$$
and their complex conjugates
$$
\overline{\bw}_{i,j}=\sum_{h=1}^{2g} u_{i,j}^{h} \overline{\per}_h.
$$
Therefore the $u_{i,j}^{h}$ are real-valued. They are sometimes called Betti coordinates of an abelian logarithm.

We then define the function 
\begin{align*}
\Gamma :   I \subseteq \R^2 & \rightarrow  (\R^{nr})^{2g}\\
s & \mapsto   (u_{i,j}^{h}(s)).
\end{align*}
This is a real analytic function and $Z=\Gamma (D)$ is a subanalytic set in $(\R^{nr})^{2g}$, therefore definable in the o-minimal structure $\R_{\text{an}}$.

Now, if $s_0 \in D(T)$, there exist $2g$ integers $b_1, \dots, b_{2g}$ with
\begin{equation}\label{relD1}
\sum_{i,j}a_{i,j} \bw_{i,j}(s_0) = b_1 \per_1(s_0) + \dots + b_{2g} \per_{2g}(s_0).
\end{equation}

Since the $\per_h(s_0)$ are $\R$-linearly independent, we have
\begin{equation}\label{linrel}
\sum_{i,j}a_{i,j }u_{i,j}^{h}(s_0) =b_h, \quad \text{for $h=1, \dots , 2g$.}
\end{equation}

We now consider $(u_{i,j}^{h})$ as real coordinates on $Z$.

Now, for $T\geq 1$ we define
$$
Z(T)=\{ (u_{i,j}^{h})  \in Z: \eqref{linrel}  \text{ for some $(a_{i,j},b_h) \in \Z^{rn+2g} \setminus \{0\}$ with } |a_{i,j}|, |b_h|\leq T   \}.
$$

\begin{lemma}\label{lemunibound}
	For every choice of $a_{i,j },b_h \in \R$, not all zero, the subset of $Z$ for which \eqref{linrel} holds is finite.
\end{lemma}

\begin{proof}
	By contradiction suppose that the subset of $Z$ of points satisfying \eqref{linrel} for some choice of coefficients is infinite. This 
	would imply that there exists an infinite set $D'\subseteq D$ on which for every $s_0 \in D'$,
	$$
	\sum_{i,j}a_{i,j} A_j \bz_{i}(s_0)= b_1 \per_1(s_0) + \dots + b_{2g} \per_{2g}(s_0).
	$$
	Since this relation holds on a set with an accumulation point, it must hold on the whole $D$ (see Ch. III, Theorem 1.2 (ii) of \cite{LangCompl}), contradicting Lemma \ref{lemtransc}.
\end{proof}

\begin{lemma}\label{lemeps}
	Under the above hypotheses, for every $\epsilon >0$, we have $|Z(T)|\ll_\epsilon T^\epsilon$, for every $T\ge 1$.
\end{lemma}

\begin{proof}
	Define
	$$
	W=\lg  \left(\alpha_{i,j},\beta_h , u_{i,j}^{h}\right)  \in  (\R^{rn} \setminus \{0\})\times \R^{2g}\times Z: \sum_{i,j}\alpha_{i,j }u_{i,j}^{h} =\beta_h, \quad \text{for $h=1, \dots , 2g$} \rg,
	$$
	and recall \eqref{def}.
We denote by $\pi_1$ the projection on the first $rn+2g$ coordinates and by $\pi_2$ the projection onto $Z$. Then, we have $Z(T)\subseteq \pi_2(W^{\sim}(\Q,T))$ and therefore $|Z(T)| \leq |\pi_2(W^{\sim}(\Q,T))|$. We claim that $|\pi_2(W^{\sim}(\Q,T))|\ll_\epsilon T^\epsilon$. Suppose not; then by Proposition \ref{HabPila} there exists a continuous definable $\delta:[0,1] \rightarrow W$ such that $\delta_1:=\pi_1 \circ \delta : [0,1] \rightarrow \R^{rn+2g}$ is semi-algebraic and the composition $\delta_2:=\pi_2 \circ \delta : [0,1] \rightarrow Z$ is non-constant. Therefore, there is a connected infinite subset $E\subseteq [0,1]$ such that $\delta_1(E)$ is a segment of an algebraic curve and $\delta_2(E)$ has positive dimension. Let $D'$ be an infinite connected subset of $D$ with $ \Gamma (D')\subseteq \delta_2(E)$. The coordinate functions $\alpha_{i,j},\beta_h$ on $D'$ satisfy $rn+2g -1$ independent algebraic relations with coefficients in $\C$. Moreover, we have the relations given by
	$$
	\sum_{i,j}\alpha_{i,j} \bw_{i,j}= \beta_1 \per_1 + \dots + \beta_{2g} \per_{2g},
	$$
	which, as $\bw_{i,j}=A_j \bz_{i}$, translate to 
	$$
	\sum_{i,j}\alpha_{i,j} A_j \bz_{i}= \beta_1 \per_1 + \dots + \beta_{2g} \per_{2g}.
	$$
	Thus, recalling that the $\alpha_{i,j }$ cannot all be 0 and that the $A_j$ have entries in $K$ by Lemma \ref{lementrM}, we have $g$ algebraic relations among the $\alpha_{i,j},\beta_h,\bz_i$ over $F=K(\per_1, \dots , \per_{2g})$.
	
	Then, on $D'$, and therefore on the whole $D$, the $nr+2g+ng$ functions $\alpha_{i,j},\beta_h,\bz_i$ satisfy $nr+2g-1+g$ independent algebraic relations over $F$. Thus, since by assumption $g>1$, we have
	$$
	\textsl{tr.deg.}_F F (\bo{z}_1, \dots ,\bo{z}_n )\leq ng-g+1 <ng,
	$$
	contradicting Lemma \ref{lemtransc}, and proving the claim.
\end{proof}

\begin{proof}[Proof of Proposition \ref{propest}]
	Since the $\per_h$ are $\R$-linearly independent, if $s_0 \in D$ satisfies \eqref{relD}, then the equations \eqref{linrel} hold for $\Gamma (s_0)$ for some integers $b_1, \dots , b_{2g}$. Now, since $D$ is a compact subset of $\R^2$, each $\bz_i(D)$ is bounded and therefore, if $\bz_{1}(s_0), \dots ,\bz_{n}(s_0),$ 
	$\per_1(s_0),\dots ,\per_{2g}(s_0)$ satisfy \eqref{relD1}, then $|b_1|,\dots , |b_{2g}|$ are also bounded in terms of the $|a_{i,j}|$ and thus of $T$. Therefore, $\Gamma (s_0) \in Z(\gamma_1 T)$ for some $\gamma_1$ independent of $T$.
	Now, using Proposition \ref{unifbound}, Lemma \ref{lemunibound} and the fact that $\Gamma $ is a definable function, we see that there exists a $\gamma_2$ such that, for any choice of $a_{i,j}$ and $ b_h$, there are at most 
	$\gamma_2$ elements $s_0$ in $D$ such that $\bz_{1}(s_0), \dots ,\bz_{n}(s_0)$, $\per_1(s_0),\dots,\per_{2g}(s_0)$ satisfy \eqref{relD1}. Thus $|D(T)|\ll |Z(\gamma_1 T)|$ and the claim of Proposition \ref{propest} follows from Lemma \ref{lemeps}.
\end{proof}


\section{Relations on a fixed abelian variety} \label{fixedrel}


Let $G$ be an abelian variety of dimension $g$ defined over a number field $L$ and let $T(G)$ be its tangent space at the origin. The N\'eron-Severi group NS$(G)$ of $G$ can be identified with the group of Riemann forms on $T(G)\times T(G)$. The degree $\deg H$ of an element $H$ of $\textrm{NS}(G)$ is defined to be the determinant of the imaginary part Im$(H)$ of $H$ on the period lattice $\Lambda(G)$ of $G$. Suppose we are given an ample symmetric divisor $\D$ on $G$ with corresponding Riemann form $H_\D$ on $T(G)\times T(G)$ of some degree $l$.
In this section, the constants $\gamma_1, \gamma_2, \dots $ and the ones implied in the $\ll$ notation depend on $l$ and $g$. 

We indicate by $h_F(G)$ the stable Faltings height of $G$ taken with respect to a sufficiently large field extension of $L$ so that $G$ has at least semistable reduction everywhere.
Moreover, the divisor $\D$ induces a N\'eron-Tate height $\widehat{h}_\D$ on $G$.

Suppose $Q_1,\dots ,Q_m \in G(L)$ are $m$ points dependent over $\Z$, with $\widehat{h}_{\mc{D}}(Q_i)\leq q$ for some $q\geq 1$. Define
$$
\mathcal{L}(Q_1, \dots , Q_m)=\{ (a_1, \dots , a_m) \in \Z^m: a_1Q_1+\dots +a_mQ_m= O \}.
$$
This is a sublattice of $\Z^m$ of some positive rank $r$. We want to show that $\mathcal{L}(P_1, \dots , P_m)$ has a set of generators with small max norm 
$|\bo{a}|=\max \{|a_1|, \dots ,|a_m|\}$. 

\begin{proposition}\label{lemgeneratorstau}
Under the above hypotheses, there are generators $\bo{a}_1, \dots , \bo{a}_r$ of the lattice $\mathcal{L}(Q_1, \dots , Q_m)$ with
$$
|\bo{a}_i|\ll  \kappa^{\gamma_1} q^{\frac{1}{2}(m-1)} (h_F(G)+\gamma_2)^{\gamma_3},
$$
for some positive $\gamma_1, \gamma_2,\gamma_3$, where $\kappa=[L:\Q]$.
\end{proposition}

For the case $g=1$, this is Lemma 6.1 of \cite{BC2016}.

We need a few auxiliary results in order to prove the above proposition. 
Some of the arguments we are going to use, in particular in the proof of Lemma \ref{lemheighttor}, were suggested to us by David Masser. \\

We first need to see how to associate to $G$ a principally polarized abelian variety.
Let $S_g$ be the Siegel space of $g\times g$ symmetric matrices of positive definite imaginary part. Let $\tau \in S_g$ and set $\Lambda=\Z^g + \Z^g \tau$. The analytic space $\C^g/\Lambda$ embeds in some projective space $\mathbb{P}^N$ via a function $z\mapsto \Theta_\tau(z)$ whose coordinates are given in (1), p.~510 of \cite{David93}.
The image of such function is an abelian variety $A(\tau)$ which is principally polarized with associated Riemann form $H_\tau$ defined by $\tau$. In case $A(\tau)$ is defined over a number field we indicate by $\widehat{h}_{H_\tau}$ the N\'eron-Tate height relative to $H_\tau$ or, more precisely, to the ample symmetric divisor $\mathcal{D}_{H_\tau}$ associated to $H_\tau$.

Moreover, if $\tau$ is such that $\Theta_\tau(0) \in \mathbb{P}^N(\Qbar)$ we indicate by $h_\Theta(A(\tau))$ the Weil height of the point $\Theta_\tau(0)$ and call it the Theta-height of $A(\tau)$. The following lemma is a consequence of Corollary 1.3 of \cite{Pazu}.

\begin{lemma}\label{lemheight} Suppose $A(\tau)$ is defined over a number field.
There are positive constants $\gamma_4$, $\gamma_5$, $\gamma_6$, $\gamma_7$ and $\gamma_8$, depending only on $g$, such that
$$
h_\Theta(A(\tau))\leq \gamma_4 (h_F(A(\tau))+\gamma_5)^{\gamma_6}
$$
and
$$
h_F(A(\tau))\leq \gamma_7 (h_\Theta(A(\tau))+1)^{\gamma_8}.
$$
\end{lemma}

We need a result of Masser and W\"ustholz from \cite{MassWu}. This explains how to associate an $A(\tau)$ to any abelian variety via an isogeny.

Let $H$ be a Riemann form on $G$. If $G'$ is an abelian subvariety of $G$, because of Lemma 1.1 of \cite{MassWu}, we can take
\begin{equation}\label{eqdeg}
\left(\deg_{H} G'\right)^2=(\dim G' !)^2 \det \text{Im}\left( H_{|_{\Lambda(G')}}\right)
\end{equation}
as a definition of the degree $\deg_{H} G'$ of $G'$ with respect to $H$, where $\Lambda(G')=T(G')\cap \Lambda(G)$ and $H_{|_{\Lambda(G')}}$ is the restriction of the Riemann form $H$ to $\Lambda(G')$. If the divisor $\mc{D}_{H}$ associated to $H$ is very ample, then $\deg_{H} G'$ coincides with the degree of $G'$ in any projective embedding of $G$ associated to $\mc{D}_{H}$. For more details about this see p. 408 of \cite{MassWu} and p. 238 of \cite{LangNTIII}.

The following lemma associates an $A(\tau)$ to any abelian variety $G$ via an isogeny.
 
\begin{lemma}[\cite{MassWu}, Lemma 4.3]\label{lemMW}
Suppose $G$ is an abelian variety of dimension $g$ defined over $L$ and let $H$ be a positive definite element of NS$(G)$ of degree $\delta$. Then there exist $\tau \in S_g$ and an isogeny $f$ from $G$ to $A(\tau)$ of degree $\sqrt{\delta}$ with $f^* H_\tau= H$. Further, $A(\tau)$ is defined over an extension of $L$ of relative degree $\ll \delta^g$.
\end{lemma} 

The following theorem of David gives the core of the proof of Proposition \ref{lemgeneratorstau}.

\begin{theorem}[\cite{David93}, Th\'eor\`eme 1.4]\label{David}
Suppose that $A(\tau)$ is defined over a number field $L$ of degree $\kappa\geq 2$ and set $h=\max \{ 1,h_\Theta(A(\tau)) \}$. There are two positive constants $\gamma_9, \gamma_{10}$, depending only on $g$, such that any $P \in A(\tau)(L)$ satisfies one of the following two properties:
\begin{enumerate}
\item \label{caso1} there exists an abelian subvariety $B$ of $A(\tau)$, with $B\neq A(\tau)$, of degree at most $$\gamma_9 \rho (A(\tau),L)^g (\log \rho (A(\tau),L))^g$$ such that $P$ has order at most $$\gamma_9  \rho (A(\tau),L)^g \log \rho (A(\tau),L)^g$$ modulo $B$,
\item  \label{caso2} we have $$\widehat{h}_{H_\tau}(P)\geq \gamma_{10} \rho (A(\tau),L)^{-4g-2}(\log 2 \rho (A(\tau),L))^{-4g-1}h,$$
\end{enumerate} 
where
$$
 \rho (A(\tau),L)=\frac{\kappa(h+\log \kappa)}{\Vert \textup{Im} \tau \Vert}+ \kappa^{1/(g+2)}.
$$
Here $\Vert \cdot \Vert$ denotes the sup norm on the space on $g\times g$ matrices with its canonical basis.
\end{theorem}

Note that $\Vert \textup{Im} \tau \Vert \geq \sqrt 3 /2$. Therefore, we can and do ignore that factor in applying Theorem \ref{David}.

We are now ready to prove the following lemma.

\begin{lemma}\label{lemheighttor}
Suppose that $G$ is defined over a number field $L$ of degree $\kappa \geq 2$ and let $\mc{D}$ be an ample symmetric divisor on $G$ corresponding to a Riemann form $H_\D$ of degree $l$. Then, for every non-torsion $P \in G(L)$ we have
\begin{equation}\label{heightb}
\widehat{h}_{\mc{D}}(P) \gg \kappa^{-\gamma_{11}} (h_F(G)+\gamma_{12})^{-\gamma_{13}},
\end{equation}
while if $P \in G(L)_{\textup{tor}}$ then its order is
\begin{equation}\label{ordb}
\ll \kappa^{\gamma_{14}} (h_F(G)+\gamma_{12})^{\gamma_{15}},
\end{equation}
for some positive $\gamma_{11}, \dots, \gamma_{15}$ depending only on $g$ and $l$ and independent of $G$, $L$, $\kappa$ and $\mc{D}$.
\end{lemma}

The inequality \eqref{ordb} has been originally proved by Masser and Zannier in Lemma 7.1 of \cite{MasserZannier15} for $g=2$ and in Proposition 7.1 of \cite{MZpreprint} in general, obtaining explicit values for the exponents, namely $\gamma_{14}=\gamma_{15}=8^g g!^2$ when $G$ is principally polarized. Moreover, recently and independently Bosser and Gaudron (\cite{Bosser18}, Theorem 1.3) proved inequality \eqref{heightb} while R\'emond (\cite{Remond18}, Proposition 2.9) showed inequality \eqref{ordb}.


\begin{proof}
First, let us apply Lemma \ref{lemMW} and \ref{lemheight} to associate an $A(\tau)$ to $G$. We have an isogeny $f$ of degree $\sqrt{l}$ between $G$ and $A(\tau)$ with $f^*H_\tau =H_{\mc{D}}$ and everything is defined over an extension of $L$ of degree $\ll l^g$. Therefore we can reduce to proving the bounds for $Q=f(P)$ because isogenies change the order of a torsion point, the N\'eron-Tate height of a point and the Faltings height of the abelian variety by bounded factors depending polynomially only on their degree, see \cite{Hindry2000}, Theorem B.5.6 and \cite{MassWu}, (7.2) on p.~436.

We proceed by induction on the dimension $g$ of $A(\tau)$. For $g=1$, there is no proper non-zero abelian subvariety of an elliptic curve, therefore if $Q$ is not torsion it must satisfy the height inequality in (\ref{caso2}) of Theorem \ref{David}. If $Q \in A(\tau)(L)_{\textrm{tor}}$, then (\ref{caso1}) is true with $B=0$ and the claim follows from Lemma \ref{lemheight}.

Now suppose $g>1$. If $Q$ is non-torsion and satisfies the inequality in (\ref{caso2}) or if $Q$ has finite order and $B=0$ we are done. If this is not the case, then (\ref{caso1}) must hold with $B$ of positive dimension. Therefore, there is some positive integer $e$ with
$$
e \leq \gamma_9 \rho (A(\tau),L)^g \log \rho (A(\tau),L)^g,
$$
such that $eQ$ lies in $B$, a proper non-zero abelian subvariety of $A(\tau)$ of degree 
\begin{equation}\label{ineqdelta}
\Delta\leq \gamma_9 \rho (A(\tau),L)^g (\log \rho (A(\tau),L))^g.
\end{equation}
In both cases, if we find lower bounds of the form \eqref{heightb} and \eqref{ordb} for $eQ$, then we are done as, by standard properties of the N\'eron-Tate height, we have
$$
\widehat{h}_{H_\tau}(Q)=e^{-2}\widehat{h}_{H_\tau}(eQ)\geq \left(\gamma_9 \rho (A(\tau),L)^g (\log \rho (A(\tau),L)^g) \right)^{-2} \widehat{h}_{H_\tau}(eQ),
$$
and the order of $Q$ is at most $e$ times the order of $eQ$.

As before, we consider the Riemann form $H_\tau$ associated to $\tau$ as an element of $\text{NS}(A(\tau))$. By restricting $H_\tau$ to $B$ we get an element of $ \text{NS}(B)$ of degree $\Delta^2/(\dim B!)^2$.
Now we use Lemma \ref{lemMW}. Suppose $B$ has dimension $g'$. Then, there exist $\nu \in S_{g'}$ and an isogeny $f'$ from $B$ to $A'(\nu)$ of degree $\Delta/g'!$ such that  $f'^*H_\nu=H_{\tau|_B}$. Moreover, we have that $A'(\nu)$ is defined over an extension of $L$ of degree $\ll \Delta^{2g'}$.

If $eQ$ has infinite order, by Theorem B.5.6 (d) of \cite{Hindry2000} and the inductive hypothesis and  we have
$$
\widehat{h}_{H_\tau}(eQ)= \widehat{h}_{H_\nu}(f'(eQ)) \gg \left(\kappa \Delta^{2g'}\right)^{-\gamma_{16}} \left( h_F(A'(\nu))+1\right)^{-\gamma_{17}}.
$$
On the other hand, if $eQ$ is torsion, then its order is bounded by the order of $f'(eQ)$ times the degree of $f'$. Therefore, by the inductive hypothesis we have that $f'(eQ)$ has order at most
$$
\gamma_{18} (\kappa \Delta^{2g'})^{\gamma_{19}} (h_F(A'(\nu))+1)^{\gamma_{20}}.
$$

To conclude, we want to bound $h_F(A'(\nu))$ with a polynomial in $h_F(A(\tau))$ and $\Delta$. Since $A'(\nu)$ and $B$ are isogenous we have, by (7.2) on p.~436 of \cite{MassWu},
$$
h_F(A'(\nu)) \ll h_F(B)+ \frac{1}{2}\log (\Delta /(g'!)) + 1.
$$
By Lemma 1.4 of \cite{MassWu}, there is an isogeny of degree at most $(\Delta/g'!)^2$ from $B \times B^\perp$ to $A(\tau)$, where $ B^\perp$ is the abelian subvariety of $A(\tau)$ orthogonal to $B$ with respect to the Riemann form $H_\tau$. The dual isogeny from $A(\tau)$ to $B \times B^\perp$ has degree at most $(\Delta/g'!)^{4g-2}$ (see \cite{Hindry2000}, Remark A.5.1.6). Combining this with (7.2) on p.~436 of \cite{MassWu}, we have
$$
h_F(B)+h_F(B^\perp) \ll h_F(A(\tau))  +\log (\Delta^2/(g'!)^2) + 1.
$$
We can forget about the term $h_F(B^\perp)$ since there exists a lower bound for the value of the Faltings height which depends only on the dimension (see Remark 1.4 in \cite{Pazu}).
Finally, we obtain the claim combining these last two estimates and \eqref{ineqdelta}.

\end{proof}

We are now ready to prove Proposition \ref{lemgeneratorstau}. Suppose first that the points $Q_1, \dots , Q_m$ are all torsion. The result easily follows from Lemma \ref{lemheighttor}.

Now suppose not all points are torsion. Then, by Theorem A of \cite{Masser88}, we have
$$
|\bo{a}_i|\leq n^{n-1}\omega \left( \frac{q}{\eta} \right)^{\frac{1}{2}(n-1)},
$$
where $\omega$ is the cardinality of $G(L)_{\text{tor}}$, $\eta= \inf \widehat{h}_{\mc{D}}(P)$ for $P \in G(L) \setminus G(L)_{\text{tor}}$, and recall that $q$ is an upper bound for the height of the $Q_i$. Again, the claim follows from Lemma \ref{lemheighttor}. 


\section{Proof of Theorem \ref{maintheorem}} \label{proofmain}

First, we reduce to the case in which $\cB $ is principally polarized. By Corollary 1 on p.~234 of \cite{Mumford}, the generic fiber $\cB_\eta$ is isogenous to a principally polarized $\cB'_\eta$, which extends to an abelian scheme $\cB'$ over a curve $S'$ that covers $S$. By Lemma \ref{lemisog} we can prove the Theorem for the irreducible components of $f^{-1}(\cC)$ in $\cA'$. We then just suppose that $\cB $ is a principally polarized abelian scheme.

The principal polarization gives an ample symmetric divisor $\D$ on $\cB$ and an embedding in some $\mathbb{P}^N$ and therefore a Weil height $h_\D$ on $\cB$ and a Weil height $h_{\D_s}$ and a N\'eron-Tate height $\widehat{h}_{\D_s}$ on the fibers $\cB_s$. We can also define a height $h_\cC$ on $\cC$ by taking the maximum of the heights of the coordinates of $\cC$ in the different copies of $\cB$.

As in Section \ref{abintper}, let $\widehat{\cC}$ denote what remains from $\cC$ after removing the singular points and the points at which $\pi_{|_{\cC}}$ is ramified. Let $\cC_0$ be the set of points of $\widehat{\cC}$ such that $P_1,\dots ,P_n$ are $R$-dependent on the specialized abelian variety. Since the $P_i$ are not identically dependent, we have that $\cC_0$ consists of algebraic points.

Let $k$ be a number field over which $\cC$ is defined. Suppose also that the finitely many points we excluded from $\cC$ to get $\widehat{\cC}$, which are algebraic, are defined over $k$. 
By the above considerations we also have that $\cC$ embeds in $(\mathbb{P}^N)^n$. After removing finitely many further points (possibly enlarging $k$ again) from $\cC$ we can suppose that $\cC$ embeds in an affine space $\mathbb{A}^{nN}$ and call $\overline{\cC}$ the closure of $\cC$ in $\mathbb{A}^{nN}$.

In this section, the constants depend on $\cB$, $\cC$ and on the choices (e.g. of polarization) we have made until now.

Now, by Silverman's Specialization Theorem (\cite{Sil83}, Theorem C) our set $\cC_0$ projects via $\pi$ to a set of bounded height. Therefore, there exists a positive $\gamma_1$ with
\begin{equation}\label{boundedheight}
h_\cC(\bo{c}_0)\leq \gamma_1,
\end{equation}
for all $\bo{c}_0 \in \cC_0$. 

We see now a few consequences of this bound.
If $\delta>0$ is a small real number, let us call
\begin{equation*} 
\cC^{\delta}=\lg \bo{c} \in \cC: |\bo{c}| \leq \frac{1}{\delta}, |\bo{c}-\bo{c}^*|\geq  \delta \mbox{ for all $\bo{c}^* \in \overline{\cC}\setminus \widehat{\cC}$} \rg .
\end{equation*}

Here, $|\cdot|$ indicates the max norm induced by the embedding in $\mathbb{A}^{nN}$.
\begin{lemma}\label{lemconj}
There is a positive $\delta$ such that there are at least $\frac{1}{2}[k(\bo{c}_0):k]$ different $k$-embeddings $\sigma$ of $k(\bo{c}_0)$ in $\C$ such 
that $\bo{c}_0^\sigma$ lies in $\cC^{\delta}$ for all $\bo{c}_0 \in \cC_0$. 
\end{lemma}

\begin{proof}
The bound \eqref{boundedheight} implies that the coordinates of all $\bo{c}_0$ have bounded height. Then, the claim follows as in Lemma 8.2 of \cite{MZ14a}.
\end{proof}

\begin{lemma}\label{lemdif}
There exists a positive constant $\gamma_2$ such that, for every $\bo{c}_0 \in \cC_0$ and every $i=1, \dots , n$, we have
$$
\widehat{h}_{\D_{\pi(\bo{c}_0)}}(P_i(\bo{c}_0))\leq \gamma_2.
$$
\end{lemma}

\begin{proof}
We have $h_{\D_{\pi(\bo{c}_0)}}(P_i(\bo{c}_0)) \leq h_\cC(\bo{c}_0)$ and, by a result of Silverman and Tate (Theorem A of \cite{Sil83}), we have $\widehat{h}_{\D_{\pi(\bo{c}_0)}}(P_i(\bo{c}_0))\leq h_{\D_{\pi(\bo{c}_0)}}(P_i(\bo{c}_0)) +\gamma_5(h_\cC(\bo{c}_0)+1)$. The claim follows from (\ref{boundedheight}).
\end{proof}


Now, by Northcott's Theorem \cite{Northcott1949} and \eqref{boundedheight}, in order to prove Theorem \ref{maintheorem} it is sufficient to bound the degree of the points in $\cC_0$ to prove finiteness.

Fix $\bo{c}_0 \in \cC_0$ and let $d_0=[k(\bo{c}_0):k]$ which we suppose to be large. First, by Lemma \ref{lemconj}, we can choose $\delta$, independently of $\bo{c}_0$, such that at least half of its conjugates lies in $\cC^\delta$. Since $\cC^\delta$ is compact, there are $\bo{c}_1, \dots , \bo{c}_{\gamma_4}\in \widehat{\cC} $ with corresponding neighborhoods $N_{\bo{c}_1}, \dots , N_{\bo{c}_{\gamma_4}}$, and $D_{\bo{c}_1}, \dots , D_{\bo{c}_{\gamma_4}} \subseteq \pi(\widehat{\cC})$, where $D_{\bo{c}_i}\subseteq \pi(N_{\bo{c}_i})$ contains $\pi(\bo{c}_i)$ and is homeomorphic to a closed disc and we have that the $\pi^{-1}(D_{\bo{c}_i})\cap N_{\bo{c}_i}$ cover $\cC^{\delta}$.

 We can then suppose that $D_{\bo{c}_1}$ contains $s_0^\sigma=\pi(\bo{c}_0^\sigma)$ for at least $\frac{1}{2\gamma_4}d_0 $ conjugates $\bo{c}^\sigma_0$ of $\bo{c}_0$ over $k$. Since each $s\in S(\C)$ has a uniformly bounded number of preimages $\bo{c} \in \cC(\C)$, we can suppose we have at least $\frac{1}{\gamma_5}d_0$ distinct such $s_0^\sigma$ in $D_{\bo{c}_1}$.

  Now, all such conjugates $\bo{c}_0^\sigma$ are contained in $\cC_0$ because the $P_1(\bo{c}_0^\sigma), \dots , P_n(\bo{c}_0^\sigma)$ satisfy the same relations holding between the points specialized at $\bo{c}_0$. So there are $\rho_1,\dots , \rho_n \in R$, not all zero, such that
\[
\rho_1P_1(\bo{c}_0^\sigma)+\dots +\rho_n P_n(\bo{c}_0^\sigma)=O
\]
on $\cB_{s^{\sigma}_0}$.

By Lemma \ref{lemdif}, we have
$$
\widehat{h}_{\D_{s^{\sigma}_0}}(P_i(\bo{c}_0^\sigma))\leq \gamma_2,
$$
and the $P_i(\bo{c}_0^\sigma)$ are defined over $k(\bo{c}_0^\sigma)$. 

If $\varphi_1, \dots ,\varphi_r$ are $\Z$-generators of $R$, one can write 
\begin{equation}\label{eqrho}
\rho_i=\sum_{j=1}^r a_{i,j}\varphi_j,
\end{equation}
 for some integers $a_{i,j}$. Now, if we call $Q_{i,j}=\varphi_j P_i(\bo{c}_0^\sigma)$, we have that the $Q_{i,j}$ are $nr$ $\Z$-dependent points in $\cB_{s^{\sigma}_0}(k(\bo{c}_0^\sigma))$ and have height
$$
\widehat{h}_{\D_{s^{\sigma}_0}} (Q_{i,j})  \ll  \widehat{h}_{\D_{s^{\sigma}_0}}(P_i(\bo{c}_0^\sigma))\ll 1.
$$

The divisor $\mc{D}_{s_0^\sigma}$ corresponds to a principal polarization and so to a degree 1 Riemann form on $T(\cB_{s_0^{\sigma}})\times T(\cB_{s_0^{\sigma}})$. Therefore, we can apply Proposition \ref{lemgeneratorstau} and we can suppose that for the coefficients $a_{i,j}$ in \eqref{eqrho} we have $| a_{i,j}|\ll [k(\bo{c}_0^\sigma):\Q]^{\gamma_6}  (h_F(\cB_{{s^{\sigma}_0}})+\gamma_7)^{\gamma_8}$.
Moreover, $h_F(\cB_{{s^{\sigma}_0}})\ll h_\cC(\bo{c}_0^\sigma)\ll 1$ (see, e.g., the discussion on p.~123 of \cite{David91}).
Thus we can suppose that all $|a_{i,j}|$ are $\ll d_0^{\gamma_6}$ and we then have that  there are at least $\frac{1}{\gamma_5} d_0$ distinct $s_0^\sigma \in D_{\bo{c}_1}(\gamma_9 d_0^{\gamma_6})$ (Recall the definition of $D(T)$ in \eqref{DefD(T)}).

By Proposition \ref{propest} we have that $|D_{\bo{c}_1}(\gamma_{9} d_0^{\gamma_6})|\ll_\epsilon d_0^{\epsilon \gamma_6}$. Therefore, if we choose $\epsilon = \frac{1}{2\gamma_6}$ we have a contradiction if $d_0$ is large enough, which proves that $d_0$ has to be bounded as required.
\section{Some consequences of Theorem \ref{thmscheme}} \label{aux}

We start by deducing from Theorem \ref{thmscheme} an apparently stronger statement, Theorem \ref{thmaux}. As before we consider an abelian scheme $\cA$ over a smooth irreducible curve $S$, of relative dimension $g$ and an irreducible curve $\cC\subseteq \cA$, not contained in a fixed fiber. Everything is defined over a fixed number field $k$.

\begin{theorem}\label{thmm}
Let $m$ be an integer, $1\leq m \leq g$. If the intersection of $\cC$ with the union of all flat subgroup schemes of $\cA$ of codimension at least $m$ is infinite, then there exists a finite cover $S'\rightarrow S$ such that $\cC\times_S S'$ is contained in a flat subgroup scheme of $\cA\times_S S'$ of codimension at least $m-1$.
\end{theorem}


\begin{proof}
We can suppose $m\geq 3$. Note that, if the hypothesis is true for $\cC$ in $\cA$ then it is true after any base change. We can therefore suppose that there is a $G$ that is the smallest component of a flat subgroup scheme of $\cA$ containing $\cC$ and that after any base change $S'\rightarrow S$, the curve $\cC\times_S S'$ is not contained in a flat subgroup scheme of smaller dimension. Note that, for any positive integer $d$, Theorem \ref{thmm} is true for $\cC$ if and only if it is true for the image of $\cC$ under the multiplication-by-$d$ map. Therefore we can suppose that $G$ is an abelian subscheme $\cB$ of $\cA$. If the codimension of $\cB$ in $\cA$ is at least $m-1$ we are done so we suppose codim$_\cA (\cB)\leq m-2$. Any point of $\cC$ lying in a component $H$ of a flat subgroup scheme of $\cA$ of codimension at least $m$ lies also in $H\cap \cB$. This is a component of a flat subgroup scheme of $\cB$. We see that it has codimension at least 2 in $\cB$:
$$
2\leq \text{codim}_\cA (H) - \text{codim}_\cA (\cB)=\text{dim} (\cB)-\text{dim} (H)\leq \text{dim} (\cB)-\text{dim} (H\cap \cB)=\text{codim}_\cB (H \cap\cB).
$$
Then $\cC$ contains infinitely many points lying in flat subgroups schemes of $\cB$ of codimension at least 2 and so, by Theorem \ref{thmscheme}, it is contained in a proper flat subgroup scheme of $\cB$ (possibly after a base change), contradicting the minimality of $\cB$.
\end{proof}

The following statement will be used for proving Theorem \ref{Pellthm}. The key point is that, by Lemma \ref{lemmat}, an abelian scheme whose generic fiber has no one-dimensional abelian subvariety cannot have codimension 1 flat subgroup schemes.

\begin{corollary}\label{cor1}
Let $\mc{A}, S$ and $\cC$ be as above. If $\cC \cap \bigcup G$, where the union runs over all flat subgroup schemes of $\cA$ which do not contain $\cC$, is infinite, then there exists a finite cover $S'\rightarrow S$ such that the generic fiber of $\cA\times_S S'$ contains a one-dimensional abelian subvariety.  
\end{corollary}

\begin{proof}
As in the proof above we can suppose that the smallest component of a flat subgroup scheme of $\cA$ containing $\cC$ is actually an abelian subscheme $\cB$ of $\cA$ and after any base change $\cC$ is not contained in a flat subgroup scheme of smaller dimension. We can also suppose that there is no base extension such that the generic fiber of $\cA$ has a one dimensional abelian subvariety if it did not have one already. Now, if $\cC \cap \bigcup G$, where the union runs over all flat subgroup schemes of $\cA$ which do not contain $\cC$, is infinite, then $\cC \cap \bigcup (\mc{B}\cap G)$ is infinite and each $\mc{B}\cap G$ considered in that union is a proper flat subgroup scheme of $\mc{B}$. But then, if $\cA$ has no one-dimensional abelian subvariety we must have codim$_\cB(\mc{B}\cap G)\geq 2$. In this case, by Theorem \ref{thmscheme}, $\cC$ is contained in a proper flat subgroup scheme of $\cB$ (possibly after a base change), contradicting the minimality of $\cB$.   
\end{proof}

We keep the above assumptions on $S$ and $\cC$ but we now assume that, for some $n\geq 1$, $\cA$ is the $n$-fold fibered power of an abelian scheme $\mc{B}\rightarrow S$ whose generic fiber has no one-dimensional abelian subvariety, even after a base change. Let $P_1, \dots, P_n\in \mc{B}_\eta(k(\cC))$ be the points defined by $\cC$. We define
$$
M=\lg (a_1, \dots , a_n) \in \Z^n: \sum_{i=1}^n a_i P_i=O \rg,
$$
to be the lattice of integral relations among the $P_i$. Moreover, for every $\bo{c}\in  \cC(\C)$, we let
 $$
\Lambda (\bo{c}) = \lg (a_1, \dots , a_n) \in \Z^n: \sum_{i=1}^n a_i P_i(\bo{c}) =O\rg.
 $$
 Then, we must have $\Lambda (\bo{c})\supseteq M$ for all $\bo{c} \in \cC(\C)$. 
 
 The following result is a  direct consequence of Corollary \ref{cor1}.
 
\begin{corollary}\label{ausxthm}
We have $\Lambda (\bo{c})= M$ for all except at most finitely many $\bo{c} \in \cC(\C)$.
\end{corollary}

\begin{proof}
Each point $\bo{c}$ such that $\Lambda (\bo{c}) \neq M $ is contained in a flat subgroup scheme of $\cA$ which does not contain $\cC$. Therefore, there can only be at most finitely many such points by Corollary \ref{cor1}.
\end{proof}

Now, we present another application of our Theorem \ref{thmscheme}. Recently, Ghioca, Hsia and Tucker \cite{Ghioca17} proved a statement in the spirit of  unlikely intersections which is relatively similar to the main result of \cite{BC2017}. 
\begin{theorem}[\cite{Ghioca17}, Theorem 1.1]\label{thmHGT}
	Let $\pi_i:\mc{E}_i\rightarrow S$ be two elliptic surfaces over a curve $S$ defined over $\Qbar$ with generic fibers $E_i$, and let $\sigma_{P_i},\sigma_{Q_i}$ be sections of $\pi_i$ (for $i=1,2$) corresponding to points $P_i,Q_i \in E_i(\Qbar (S))$. If there exist infinitely many $s \in S(\Qbar)$ for which there exist some $m_{1,s}, m_{2,s}\in \Z$ such that $m_{i,s} \sigma_{P_i}(s)=\sigma_{Q_i}(s)$ for $i=1,2$, then at least one of the following properties hold:
	\begin{enumerate}
		\item there exist isogenies $\varphi:E_1\rightarrow E_2$ and $\psi:E_2\rightarrow E_2$ such that $\varphi(P_1)=\psi (P_2)$.
		\item for some $i \in \{ 1,2\}$, there exist $k_i\in \Z$ such that $k_iP_i=Q_i$ on $E_i$.
	\end{enumerate}
\end{theorem} 

In the Appendix we will see how the main theorems of \cite{BC2016} and \cite{BC2017} imply Theorem \ref{thmHGT} while in what follows we deduce an analogous statement for families of abelian varieties without elliptic factors. 

Let $\cB$ be as above and let $\cA$ be, $n+1$-fold fibered power of $\mc{B}$, for some $n\geq 1$. Let $\cC \subseteq \cA$ be an irreducible curve, as usual not contained in a fixed fiber, and suppose that everything is defined over a number field $k$. The curve $\cC$ defines points $P,P_1, \dots, P_n\in \cB_\eta(k(\cC))$ and we let 
$$
\Gamma=\left< P_1,\dots ,P_n \right>=\lg Q\in \cB_\eta(k(\cC)): Q=\sum_{i=1}^n a_i P_i, \mbox{ for some } (a_1, \dots , a_n) \in \Z^n \rg
$$       
be the subgroup of $\cB_\eta(k(\cC))$ generated by $P_1,\dots ,P_n $. This will have specializations $\Gamma(\bo{c})=\left< P_1(\bo{c}), \dots ,\right.$ $\left. P_n (\bo{c})\right>\subseteq \cB_{\pi(\bo{c})}(k(\bo{c}))$, for all $\bo{c}\in \cC(\C)$. The following is a consequence of Corollary \ref{ausxthm}.

\begin{theorem}\label{thmSG}
	If $\cB_\eta$ has no one-dimensional abelian subvariety and $P(\bo{c})\in \Gamma(\bo{c})$ for infinitely many $\bo{c}\in \cC(\C)$, then $P\in \Gamma $ identically.
\end{theorem}

Note that the assumption on $\cA_\eta$ not having elliptic factors is necessary. Indeed, one can easily construct counterexamples from the fact that a non-torsion section of a (non-isotrivial) elliptic scheme specializes to a torsion point infinitely many times (see \cite{Zannier}, Notes to Chapter 3).

\section{Almost-Pell equation}

In this section we prove Theorem \ref{Pellthm}. For now, we fix $K$ to be a field of characteristic 0, and $D\in K[X]$ a squarefree polynomial of even degree $2d\geq 4$.

Consider the hyperelliptic curve defined by $Y^2=D(X)$. If we homogenize this equation, we obtain a projective curve which is singular at infinity. There exists however a non-singular model $H_D$ with two points at infinity which we denote by $\infty^+$ and $\infty^-$. We fix them by stipulating that the function $X^d\pm Y$ has a zero at $\infty^{\pm}$. The curve $H_D$ is then a hyperelliptic curve of genus $d-1$.

Let us denote by $J_D$ its Jacobian variety, i.e. the abelian group
\[
\mbox{Jac}(H_{D})=\mbox{Div}^0(H_{D})/\mbox{Div}^P(H_{D}),
\]
where $\mbox{Div}^0(H_{D})$ denotes the groups of divisors of degree 0 and $\mbox{Div}^P(H_{D})$ is the subgroup of principal divisors. If $\Delta \in \mbox{Div}^0(H_{D})$, we will denote by $[\Delta]$ its image in the Jacobian $J_D$.

Fix a non-zero $F$ in $K[X]$ and consider the ``almost-Pell equation'' 
\begin{equation}\label{almpell}
A^2-D B^2=F,
\end{equation}
where we look for solutions $A,B \in \Kbar[X]$ with $B\neq 0$.

Suppose $F$ factors $F(X)=\beta(X-\alpha_1)^{a_1} \ldots (X-\alpha_m)^{a_m}\in K[X]\setminus \{0\}$, with $\beta \in K\setminus \{0\}$, $\alpha_1, \dots , \alpha_m \in K$ pairwise distinct and $a_1, \dots , a_m$ non-negative integers.

We order the roots of $F$ so that $D$ does not vanish at $\alpha_i$ for $i=1, \ldots, h$ and $D$ vanishes at $\alpha_i$ for $i=h+1, \ldots, m$. Note that $h$ is allowed to be 0 or $m$.
 
If $\alpha_i$ is not a common root of $D$ and $F$, then there are two points on $H_D$ with first coordinate equal to $\alpha_i$ which we denote by $\alpha_i^+$ and $\alpha_i^-$. In case $D$ and $F$ have a common root $\alpha_{i}$, then there is only one point with first coordinate $\alpha_{i}$ and we call it $\alpha_{i}$, as well. Let us assume that all these points and the two points at infinity are defined over $K$.

We now define $P_i=[\alpha_i^+-\infty^-]$ for all $i=1,\dots , h$ and $P_i=[\alpha_i-\infty^-]$ for all $i=h+1,\dots , m$ and $Q=[\infty^+-\infty^-]$ as points on $J_D(K)$. Note that, for $i\leq h$, we have $[\alpha_i^+-\infty^-]=-[\alpha_i^--\infty^+]$
since div$(X-\alpha_i)=\alpha_i^++\alpha_i^- - \infty^+ - \infty^-$ while for $i>h$, as div$(X-\alpha_{i})=2 \alpha_{i} -\infty^+-\infty^-$, we have $[\alpha_{i}-\infty^-]=-[\alpha_{i}-\infty^+]$. 

\begin{remark}\label{rem}
Suppose there is an $\alpha \in K$ that is a root of $D$ (a single root since $D$ is squarefree) and a multiple root of $F$. Then, if we have an equation like \eqref{almpell}, we must have that $(X- \alpha)^2$ divides both $A^2$ and $ B^2$. Therefore, in this case, $A^2-DB^2=F$ has a non-trivial solution if and only if $A^2 - D B^2=F/(X-\alpha)^2$ has a non-trivial solution. Thus, we can always suppose without loss of generality that, if $D$ and $F$ have a common root $\alpha$, then the order of vanishing of $F$ in $\alpha$ is equal to $1$. 
\end{remark}

In what follows we will use the fact that a function in $K(H_D)$ has the form $R+YS$ for some $R,S \in K(X)$ and that $\text{ord}_P(R+YS)=\text{ord}_{\iota(P)}(R-YS)$, where $\iota$ is the involution $Y\mapsto -Y$. Therefore, we have that $R=0$ or $S=0$ if and only if the divisor of $R+YS$ is invariant under $\iota$. 

The next two lemmas translate the existence of a non-trivial solution of the equation \eqref{almpell} to a relation between points of $J_D$ and vice versa.
For the case $F$ of degree at most 1 see Lemma 10.1 and 11.1 of \cite{MasserZannier15}.

\begin{lemma}\label{lemP1}
Suppose there are $A,B\in K[X]$ such that $A^2-DB^2=F$ with $B \neq 0$. Then, there exist $g_1, \dots, g_m,l\in \Z$, not all zero, with $|g_i|\leq a_i$ and $g_i \equiv a_i \mod 2$, such that
$$
\sum_{i=1}^m g_i P_i + l Q=O
$$
on $J_D$.
\end{lemma}

\begin{proof}
We consider the non-constant functions $f^\pm=A\pm YB$ on $H_D$. Since $Y^2=D$, we have that $f^+f^-=F$. Therefore, there are non-negative integers $b^+_1,b_1^-, \dots , b_h^+,b_h^-$ with ord$_{\alpha_i^{\pm}}(f^+)=b^\pm_i$, for $i= 1,\dots, h$. Note that, since ord$_{\alpha_i^{-}}(f^+)=$ord$_{\alpha_i^{+}}(f^-)$, we have $b_i^{+}+b_i^-=a_i$.
For $i>h$, because ord$_{\alpha_i}(f^+)$=ord$_{\alpha_i}(f^-)$, we must have ord$_{\alpha_i}(f^+)=a_i$.
Therefore, since $f^+$ cannot have other zeroes or poles at finite points, there exists an integer $\tilde{l}$ such that
$$
\text{div}(f^+)=\sum_{i=1}^{h} (b_i^+\alpha_i^{+} +b_i^-\alpha_i^{-})+ \sum_{i=h+1}^m a_i \alpha_i +\tilde{l} \infty^{+}-\left( \tilde{l}+\sum_{i=1}^{m} a_i\right)\infty^-.
$$
Let $\tilde{f}$ be the function $f^+/  \prod_{i=1}^h (X-\alpha_i)^{b_i^-} $, which is non-constant since $B\neq 0$. Then,
$$
\text{div}(\tilde{f})=\sum_{i=1}^{h} (b_i^+ -b_i^-)\alpha_i^+ + \sum_{i=h+1}^m a_i \alpha_i +\left(\tilde{l}+\sum_{i=1}^{h}b_i^- \right) \infty^+ - \left(\tilde{l}+\sum_{i=1}^{h}b_i^+ +\sum_{i=h+1}^m a_i\right)\infty^-.
$$
Then, if $l=\tilde{l}+\sum_{i=1}^{h}b_i^-$, we have
$$
\sum_{i=1}^{h} (b_i^+ -b_i^-)P_i+ \sum_{i=h+1}^{m} a_i P_i  +lQ=O
$$
on $J_D$. If the relation was trivial, we would have $\tilde f$ constant, which is not possible. This gives the claim.
\end{proof}



\begin{lemma}\label{lemP2}
Suppose there is a relation
\begin{equation} \label{lin.eq}
\sum_{i=1}^m  e_i P_i + l Q=O
\end{equation}
for some integers $e_1, \dots ,e_m,l$ not all zero. Moreover, suppose that, for all $i=h+1, \ldots, m$, the integer $e_{i}$ is odd or zero. Then, there exist $A,B\in K[X]$, with $B\neq 0$ such that
$$
A^2-DB^2=\beta \prod_{i=1}^m (X-\alpha_i)^{|e_i|},
$$
for some non-zero $\beta\in K $. 
\end{lemma}

\begin{proof}
From \eqref{lin.eq}, we have that
$$
\sum_{i=1}^h  e_i [\alpha_i^+-\infty^-] + \sum_{i=h+1}^m  e_i [\alpha_i-\infty^-]+  l [\infty^+-\infty^-]=O.
$$
This implies that there is a non-constant function $f \in K(H_D)$ with
$$
\text{div}(f)= \sum_{i=1}^h e_i \alpha_i^+ +\sum_{i=h+1}^m e_i \alpha_i +l \infty^+ - \left(l + \sum_{i=1}^m e_i\right) \infty^-.
$$
Let us now define $f^+= f \prod_{i,~e_i<0} (X-\alpha_i)^{|e_i|}$; then,
$$
\text{div}(f^+)=\sum_{\substack{i=1\\ e_i \ge 0}}^h  |e_i| \alpha_i^{+} +\sum_{\substack{i=1\\ e_i<0}}^h  |e_i| \alpha_i^{-}+ \sum_{i=h+1}^m  |e_i| \alpha_i + l' \infty^{+}- \left(l' + \sum_{i=1}^m |e_i|\right) \infty^-,
$$
where $l'=l-\sum_{i=1, ~e_i<0}^m |e_i|$. As $f^+$ is a rational function on $H_D$, there exist $A,B \in K(X)$ such that $f^+=A+YB$. Moreover, using the properties of the involution $\iota: Y \mapsto -Y$, we have that, if $f^-:=A-YB$,
$$
\text{div}(f^-)=\sum_{\substack{i=1 \\e_i\ge 0}}^h  |e_i| \alpha_i^{-} +\sum_{\substack{i=1\\e_i<0}}^h  |e_i| \alpha_i^{+}+ \sum_{i=h+1}^m  |e_i| \alpha_i + l' \infty^{-}- \left(l' + \sum_{i=1}^m |e_i|\right) \infty^+.
$$
Therefore, we have
$$
f^+f^-=\beta \prod_{i=1}^m (X-\alpha_i)^{|e_i|},
$$
for some non-zero $\beta \in K$. Finally, we have that $f^++f^-=2A$ has no pole at finite points, so $A$ is a polynomial and so must be $B$ because $DB^2=A^2-\beta \prod (X-\alpha_i)^{|e_i|}$ and $D$ is squarefree. 

Finally, let us prove that $B$ cannot be 0. Indeed, that would mean that all $e_i$ are even and that div$(f^+)$=div$(f^-)$. These two facts, together with our assumptions on $e_{h+1}, \dots , e_m$, imply that all $e_i=0$. Therefore, $l'=0$ and so also $l$ should be zero, which contradicts the hypotheses.
\end{proof}

Consider now the setting of Theorem \ref{Pellthm}. Recall we have a smooth, irreducible curve $S$ defined over a number field $k$ and  polynomials $D$ and $F$ with coefficients in $K=k(S)$. Recall moreover that we suppose that the Jacobian $J_D$ has no one-dimensional abelian subvariety.
We can consider a finite extension $K'$ of $K$ so that the points $P_i$ and $Q$ of $J_D$ associated to $D$ and $F$ are all defined over $K'$. Since $K'$ has the form $k(S')$ for some irreducible non-singular curve $S'$ with a finite cover $S'\rightarrow S$, in view of the claim of Theorem \ref{Pellthm}, we can suppose $S'=S$.

Similarly to what done is Section \ref{aux}, we define
\begin{multline*}
M=\lg (a_1, \dots, a_m)\in \Z_{\geq 0}^m: \exists~ A,B \in \Kbar[X], B\neq 0,\right.\\ \left. A^2-DB^2= (X-\alpha_1)^{a_1}\cdots (X-\alpha_m)^{a_m} \rg,
\end{multline*}
and, for all $s_0\in S(\C)$,
\begin{multline*}
\Delta(s_0)=\lg (a_1, \dots, a_m)\in \Z_{\geq 0}^m: \exists~ A,B \in \C[X], B\neq 0,\right.\\ \left. A^2-D_{s_0}B^2= (X-\alpha_1(s_0))^{a_1}\cdots (X-\alpha_m(s_0))^{a_m} \rg,
\end{multline*}
where the $\alpha_i(s_0)$ are the $\alpha_i$ specialized at $s_0$.

The claim of Theorem \ref{Pellthm} can be easily deduced from the following.

\begin{theorem}
We have $\Delta(s_0)=M$ for all but finitely many $s_0\in S(\C)$.
\end{theorem}

\begin{proof}
Suppose we have an infinite set $S_0$ of points $s_0 \in S(\C)$ such that there exist vectors $(a_1(s_0), \dots , a_m(s_0))\in \Delta(s_0)\setminus M$. 
Recall that we ordered $\{\alpha_1, \ldots, \alpha_m\}$ so that $D$ does not vanish at $\alpha_i$ for $i=1, \ldots, h$ and $D$ vanishes at $\alpha_i$ for $i=h+1, \ldots, m$. 
By Remark \ref{rem} we can choose the vectors $(a_1(s_0), \dots , a_m(s_0))\in \Delta(s_0)\setminus M$ so that, for all $i=h+1,\dots , m$, we have $a_{i}(s_0)=0$ or $1$.  Then, by Lemma \ref{lemP1}, we have that, for all $s_0 \in S_0$, 
$$
\sum_{i=1}^m g_i(s_0) P_i(s_0)+l(s_0)Q(s_0)=O,
$$
for some $g_1(s_0), \dots, g_m(s_0),l(s_0) \in \Z$, not all zero, with $|g_i(s_0)|\leq a_i(s_0)$, $g_i(s_0) \equiv a_i(s_0) \mod 2$ and $g_i(s_0)=\pm a_i(s_0)$ for all $i=h+1, \ldots, m$. By Corollary \ref{ausxthm}, after throwing away at most finitely many elements of $S_0$, we have that all of the above relations are actually identical. In other words
$$
\sum_{i=1}^m g_i(s_0) P_i+l(s_0)Q=O,
$$
for all $s_0 \in S_0$, identically on $S$. Applying Lemma \ref{lemP2}, we have that $(|g_1(s_0)|,\dots ,|g_m(s_0)|)$ are in $M$, and then clearly $(a_1(s_0), \dots , a_m(s_0))\in M$, which contradicts the existence of the above infinite set, as wanted.
\end{proof}

\section{Some examples} \label{Pell_example}

In this section we apply Theorem \ref{Pellthm} to some examples. Let $K=\overline \Q(t)$ and consider the generalized Pell equation
\begin{equation} \label{Pell.ex}
A^2-D_t B^2=F,
\end{equation}
where $D_t \in K[X]$ is the  family of polynomials defined by $D_t(X)=(X-t)(X^7-X^3-1)$ and $F(X)=4X+1 \in \overline \Q[X]$. 

The curve defined by $Y^2=D_t(X)$ has a non-singular model $H_{D_t}$ which is a hyperelliptic curve of genus $3$. As before, we denote by $J_{D_t}$ its Jacobian variety, which is an abelian variety of dimension $3$. It is easy to see that the polynomial $X^7-X^3-1$ has no multiple roots and that the Galois group of its splitting field is the permutation group $S_7$. Using Theorem 1.3 of \cite{Zarhin10},  
this implies that $J_{D_t}$ is geometrically simple and, in particular, it does not contain any one-dimensional abelian subvariety (for similar examples of families of this type see also \cite{EEHK}). 

We want now to prove that \eqref{Pell.ex} has no non-trivial solution with $A, B\in \overline K[X]$. Suppose by contradiction that the equation has a non-trivial solution. By Proposition 3.6 of \cite{VPT}, if $A,B$ are polynomials in $X$ satisfying $A^2-D_tB^2=F$ with $\deg_X(F)\le \frac{1}{2} \deg_X(D_t)-1$, then $A/B$ has to be a convergent of the continued fraction expansion of $\sqrt{D_t}$; in particular, this means that $A,B$ are polynomials in $K[X]$, i.e., the coefficients are rational functions in $t$. Clearing denominators, we have a new equation $A'^2-D_tB'^2=E^2F$ with $A',B' \in \overline \Q[t,X]$ and $E \in \Qbar [t]$. But now we have two cases: if $E\in \overline \Q$, then it is easy to see that the equation cannot have an identical solution because $D_t$ has degree $1$ in $t$ and $F$ is independent of $t$. On the other hand, if $E \in \overline \Q(t)\setminus \overline \Q$, then we can specialize to a zero $t_0$ of $E$, giving that $D_{t_0}(X)$ would be a square in $\overline \Q[X]$, which is again a contradiction. We can then apply Theorem \ref{Pellthm} to conclude that there are at most finitely many $t_0\in \C$ for which the specialized equation $A^2-D_{t_0}B^2=F$ is solvable. For example, for $t_0=0$, we have 
$$ (2X^4+1)^2-X(X^7-X^3-1)2^2=4X+1. $$
Note that the same argument using \cite{VPT} applies if we take as $D_t$ a squarefree polynomial in $\overline \Q[t,X]$ satisfying the hypotheses of Theorem \ref{Pellthm} with odd degree in $t$ and $F\in \overline{\Q}[X]$ with
$\deg_X(F)\le \frac{1}{2} \deg_X(D_t)-1$; in this case, we always have that the almost-Pell equation $A^2-D_tB^2=F$ is not identically solvable.

We also remark that, if $\deg_X(F)> \frac{1}{2} \deg_X(D_t)-1$, we cannot in general conclude that the polynomials $A,B$ have coefficients in $K$ rather than $\overline K$.  If we take for example $D_t(X)=X^6+X+t$ and $F(X)=-X^6-X$, then the almost-Pell equation has 
non-trivial solutions in $A,B \in \overline{\Q(t)}[X]$, i.e. $$ \left (\sqrt t \right )^2-(X^6+X+t)1^2=-X^6-X, $$
but it is easy to see that it cannot have a solution in $\Qbar(t)[X]$ just looking at the degrees in $t$. \\

Let us finally show with some examples that the requirement that $J_{D_t}$ contains no one-dimensional abelian subvariety is necessary to conclude the finiteness result. 

Examples of polynomials in $ \Q(t)[X]$ of degree at least six that are not identically Pellian but become Pellian for infinitely many specializations because the associated Jacobian has an elliptic factor appear on p.~2396 of \cite{MasserZannier15} and p.~3 of \cite{MZpreprint}.

Let us now consider the family of polynomials $D_t(X)=X^{12}+X^4+t \in \Q(t)[X]$ and let us take $F(X)=X^4-1$. We can define the map
$$ \beta: H_{D_t} \rightarrow H_{\widetilde{D}_t} \quad \beta(X,Y)=(X_1, Y_1)=(X^4, X^2Y), $$
where $H_{\widetilde{D}_t}$ is the genus 1 curve defined by the equation $Y_1^2=\widetilde{D}_t(X_1)=X_1^4+X_1^2+tX_1$. Let us define also $\widetilde{F}(X_1)=X_1-1$ and consider the almost-Pell equation 
\begin{equation} \label{new.almost}
 A^2-\widetilde{D}_t B^2=\widetilde{F}.
\end{equation}
Using \cite{VPT} and the same argument of the previous example, \eqref{new.almost} is not identically solvable (neither is the original equation $A^2-D_tB^2=F$). However we show that there are infinitely many $t_0\in \C$ such that \eqref{new.almost} specialized in $t_0$ has a non-trivial solution. In fact, using the notation introduced in the previous section, we consider the Jacobian $J_{\widetilde{D}_t}$ of $H_{\widetilde{D}_t}$ which can be identified with $H_{\widetilde{D}_t}$ itself by choosing a point on it. Consider moreover $P_t=[(1, \sqrt{2+t})-\infty^-]$ and $Q_t=[\infty^+ -\infty^-]$, where $\infty^+$ and $\infty^-$ are the two points at infinity of $H_{\widetilde{D}_t}$. First, notice that $Q_t$ is not identically torsion of $J_{\widetilde{D}_t}$, otherwise the polynomial $\widetilde{D}_t$ would be identically Pellian (i.e. the Pell equation $A^2-\widetilde{D}_tB^2=1$ would be identically solvable), which is not the case again using \cite{VPT}. Using Lemma \ref{lemP1} and \ref{lemP2}, we have that for some $t_0\in \C$, the equation \eqref{new.almost} has a non-trivial solution if and only if there exists an integer $l(t_0)$ such that $P_{t_0}=l(t_0)Q_{t_0}$. 

However, it is a consequence of Siegel's theorem for integral points on curves over function fields that this happens for infinitely many $t_0\in \C$. 

In fact, if $L=\Q(\sqrt{2+t})$, then both $P_t$ and $Q_t$ are defined over $L$. Suppose that we have a finite number of $t_0 \in \C$ such that $P_{t_0}=l(t_0)Q_{t_0}$ for some $l(t_0)\in \Z$ and denote by $S_0$ the set of such $t_0$. Let $R_t^{(l)}=P_t-lQ_t$. Then, as $Q_t$ is not identically torsion, the set of $R_t^{(l)}$, for varying $l$, is an infinite set of $L$-rational points of $H_{\widetilde{D}_t}$. Now, the fact that for all $l$ the set of $t_0$ such that $R_{t_0}^{(l)}=O_{t_0}$ is contained in $S_0$ implies that all $R_t^{(l)}$ are $S_0$-integral. Since we supposed that $S$ is finite, this contradicts Siegel's theorem for integral points on curves over function fields (see \cite{Silv09}, Theorem 12.1). 



 Thus we proved that there exist infinitely many $t_0\in \C$ such that \eqref{new.almost} has a non-trivial solution. For such a $t_0$ suppose we have
$$ A_1^2-(X_1^4+X_1^2+t_0X_1)B_1^2=X_1-1, $$
for some $A_1, B_1 \in \C[X_1]$.
Recalling that $X_1=X^4$, we have
$$ X^4-1=(A_1(X^4))^2-(X^{16}+X^8+t_0X^4)(B_1(X^4))^2=(A_1(X^4))^2-D_{t_0}(X)(X^2B_1(X^4))^2, $$
so $A_1(X^4), X^2B_1(X^4)$ is a solution of the original equation $A^2-D_{t_0}B^2=F$. Hence, we showed that the equation $A^2-D_t B^2=F$ is not identically solvable but there are infinitely many $t_0\in \C$ such that the specialized equation is solvable.

\section*{Fundings}

This work was supported by the European Research Council [267273], the Engineering and Physical Sciences Research Council [EP/N007956/1 to F.B. and EP/N008359/1 to L.C.], the Istituto Nazionale di Alta Matematica [Borsa Ing. G. Schirillo to L.C.] and the Swiss National Science Foundation [165525 to F.B.].

\section*{Acknowledgments}

We thank Philipp Habegger, Gareth Jones and Jonathan Pila for their support, Umberto Zannier for sharing his ideas and David Masser for suggesting to us some arguments used in Section \ref{fixedrel}. We are grateful also for sending us their preprint \cite{MZpreprint}. We moreover thank Daniel Bertrand, Gabriel Dill, Harry Schmidt, Amos Turchet and Francesco Veneziano for useful discussions.

\appendix

\section{An alternative proof of Theorem \ref{thmHGT}}

In this appendix we see how the main results of \cite{BC2016} and \cite{BC2017} (and so also our Theorem \ref{thmscheme}) give an alternative proof of a result of Hsia, Ghioca and Tucker, namely Theorem \ref{thmHGT}.

For the sake of the reader we state it here again.

\begin{theorem}[\cite{Ghioca17}, Theorem 1-1]
	Let $\pi_i:\mc{E}_i\rightarrow S$ be two elliptic surfaces over a curve $S$ defined over $\Qbar$ with generic fibers $E_i$, and let $\sigma_{P_i},\sigma_{Q_i}$ be sections of $\pi_i$ (for $i=1,2$) corresponding to points $P_i,Q_i \in E_i(\Qbar (S))$. If there exist infinitely many $s \in S(\Qbar)$ for which there exist some $m_{1,s}, m_{2,s}\in \Z$ such that $m_{i,s} \sigma_{P_i}(s)=\sigma_{Q_i}(s)$ for $i=1,2$, then at least one of the following properties hold:
	\begin{enumerate}
		\item there exist isogenies $\varphi:E_1\rightarrow E_2$ and $\psi:E_2\rightarrow E_2$ such that $\varphi(P_1)=\psi (P_2)$. \label{it1}
		\item for some $i \in \{ 1,2\}$, there exist $k_i\in \Z$ such that $k_iP_i=Q_i$ on $E_i$.\label{it2}
	\end{enumerate}
\end{theorem} 

\begin{proof}

We consider an infinite sequence $(s_n)_{n\in \N}$ of points of $ S(\Qbar)$ for which there exist some $m_{1,s_n}, m_{2,s_n}\in \Z$ such that \begin{equation}\label{eqa1}
m_{1,s_n} \sigma_{P_1}(s_n)=\sigma_{Q_1}(s_n) ,
\end{equation} and \begin{equation}\label{eqa11}
m_{2,s_n} \sigma_{P_2}(s_n)=\sigma_{Q_2}(s_n).
\end{equation} So, for each $s_n$ we have two integers $m_{1,s_n}$ and $ m_{2,s_n}$ that we assume to be the smallest in absolute value to satisfy \eqref{eqa1} and \eqref{eqa11}. We can moreover assume the absolute values of both $m_{1,s_n}$ and $m_{2,s_n}$ are unbounded as $n\rightarrow \infty$, otherwise \eqref{it2} is true. 

First, suppose $E_1 $ and $E_2$ are not isogenous. Then, by Theorem 1.3 of \cite{BC2017}, we have that $P_1$ and $Q_1$ are dependent on $E_1$, say 
\begin{equation}\label{eqa2}
a_1 P_1=b_1 Q_1,
\end{equation}
for integers $a_1, b_1$ not both zero. We have $b_1\neq 0$ otherwise $|m_{1,s_n}|$ is uniformly bounded and we have \eqref{it2}. Combining \eqref{eqa1} and \eqref{eqa2} we have that $\sigma_{P_1}(s_n)$ is torsion for almost all $n$. Using \cite{BC2017} again we have that $P_1$ is identically torsion, which would imply again \eqref{it2}, or 
\begin{equation}\label{eqa3}
a_2 P_2=b_2 Q_2,
\end{equation}
for integers $a_2, b_2$ not both zero. Arguing as above, we have that $\sigma_{P_2}(s_n)$ is torsion for almost all $n$. If we apply \cite{BC2017} one last time (actually \cite{MZ14a} suffices this time), we have that at least one between $P_1$ and $P_2$ is torsion. This again implies \eqref{it2} and finally gives the claim in the non-isogenous case.

If $E_1$ and $E_2$ are isogenous the proof is a bit more involved. Clearly we can suppose that $E_1=E_2$ and we have to prove that \eqref{it2} holds or that $P_1$ and $P_2$ are dependent. Using Theorem 2.1 of \cite{BC2016} we get 
\begin{equation}\label{eqa4}
a_1P_1+a_2P_2= b_1Q_1+b_2Q_2,
\end{equation}
for integers $a_1, b_1,a_2, b_2$ not all zero. If $b_1=b_2=0 $ we are done so suppose $b_1\neq 0$. Then, combining \eqref{eqa1} and \eqref{eqa4} we have that $\sigma_{P_1}(s_n),\sigma_{P_2}(s_n)$ and $\sigma_{Q_2}(s_n)$ are dependent for almost all $n$. We apply \cite{BC2016} again and obtain \begin{equation}\label{eqa5}
c_1P_1+c_2P_2= d_2Q_2,
\end{equation}
for integers $c_1, c_2, d_2$ not all zero, and we suppose $d_2\neq 0$ otherwise we are done. Combining \eqref{eqa11} and \eqref{eqa5}, we have
\begin{equation}\label{eqa6}
c_1 \sigma_{P_1}(s_n) +(c_2- d_2 m_{2,s_n} )\sigma_{P_2}(s_n)=O.
\end{equation}
This, together with \eqref{eqa1}, gives
\begin{equation}\label{eqa7}
e_1P_1+e_2P_2= f_1Q_1,
\end{equation}
with $f_1\neq 0$.
We combine this with \eqref{eqa1} again and get 
\begin{equation}\label{eqa8}
(e_1 - f_1 m_{1,s_n}) \sigma_{P_1}(s_n) +e_2 \sigma_{P_2}(s_n)=O.
\end{equation}
Now, if \eqref{eqa6} and \eqref{eqa8} are independent relations, then by \cite{BC2016} (\cite{MZ12} suffices here) $P_1 $ and $P_2$ are dependent and we are done, so we are left to show that 
$$
c_1 e_2 -(c_2- d_2 m_{2,s_n} )(e_1-f_1 m_{1,s_n})\neq 0.
$$
If not, then 
$$
 m_{1,s_n}=\frac{1}{f_1}\left( e_1 - \frac{c_1e_2}{c_2- d_2 m_{2,s_n}}  \right).
 $$
But its absolute value has to tend to infinity as $n\rightarrow \infty$ and this is impossible because $|m_{2,s_n}|$ has to tend to infinity as well.
\end{proof}

\bibliographystyle{amsalpha}
\bibliography{bibliography}

\end{document}